\documentclass[12pt]{article}

\usepackage[utf8]{inputenc}
\usepackage[T1]{fontenc}
\usepackage[english]{babel}
\usepackage{amsmath, amssymb, amsthm}
\usepackage{geometry}
\usepackage{graphicx}
\usepackage{enumitem}
\usepackage{float}
\usepackage{hyperref}
\usepackage{cleveref}
\usepackage{xcolor}
\usepackage{mathrsfs}
\usepackage{tikz}
\usepackage{esint}
\numberwithin{equation}{section}

\newtheorem{theorem}{Theorem}[section]

\newtheorem{lemma}[theorem]{Lemma}

\newtheorem{proposition}[theorem]{Proposition}
\theoremstyle{definition}

\newtheorem{remark}[theorem]{Remark}

\newcommand{\R}{\mathbb{R}}

\newcommand{\eps}{\varepsilon}
\newcommand{\divg}{\text{div}}

\def \dis {\displaystyle}

\def \om {\omega}
\def \Om {\Omega}

\def\ep{\varepsilon}

\def \ZZ {\mathbb Z}
\def \RR {\mathbb R}

\def \beq {\begin{equation}}
\def \eeq {\end{equation}}
\def \ba {\begin{array}}
\def \ea {\end{array}}
\def \bs {\bigskip}

\def \ecart {\noalign{\smallskip}}
\def\into{\int_\Om}

\title{Memory Effects in a Fluid via Homogenization Theory}

\title{Memory effects in the homogenization of a viscous fluid containing small elastic particles}
\author{ 
\footnotesize
\centerline{\begin{tabular}{c}
\large Juan Casado-D\'{\i}az, \hskip0.5cm  Antonio Rafael Hormigo-Molina.\\
\\
Dpto. de Ecuaciones Diferenciales y An\'alisis Num\'erico\\ Universidad de Sevilla
\\
jcasadod@us.es,\hskip0.5cm  ahormigo@us.es
\end{tabular}}
}

\begin{document}

\maketitle
\begin{abstract}
    This work deals with a fluid-structure interaction problem involving an incompressible viscous fluid and a periodic array of tiny elastic inclusions in a three-dimensional domain. We focus on a critical regime where the volume fraction of the inclusions tends to zero, but their density and stiffness are scaled so that they still influence the macroscopic behavior. The asymptotic regime depends on the parameter $\gamma=\lim_{\varepsilon\to 0}r_\varepsilon/\varepsilon^3\in [0,+\infty]$, where $\varepsilon$ is the period and $r_\varepsilon$ the inclusion size.

    For $\gamma=0$ the solid and fluid equations uncouple in the limit. For $\gamma\in (0,+\infty]$, we show that the macroscopic fluid velocity satisfies a Stokes system with memory: the standard Stokes equations are augmented by convolution-in-time terms arising from the elastic vibrations of the inclusions. The existence of the convolution kernels is established for both $\gamma\in(0,+\infty)$ and $\gamma=\infty$ using semigroup theory. For $\gamma=\infty$ the microscopic problem reduces to a purely conservative elastodynamic equation and the kernels can be computed explicitly via spectral decomposition. They are oscillatory and almost-periodic in time, yielding a persistent, non-decaying memory. For $\gamma\in(0,+\infty)$ the microscopic elastodynamic equation is coupled to a local Stokes problem, whose viscous dissipation yields memory kernels that decay to zero at infinity.
\end{abstract}
\par\bs\noindent

\section{Introduction}

In this paper, we show how a Newtonian viscous fluid can develop long-range memory effects. When a large number of tiny elastic inclusions are periodically suspended in a viscous fluid, the interplay between the elastic vibrations of the particles and the viscous dissipation of the surrounding medium generates convolution-in-time terms in the effective equations. Such emergent behaviors are central to the design of advanced complex fluids, such as locally resonant acoustic metamaterials, where heavy and highly deformable inclusions are used to achieve specific wave-damping or shock-absorbing properties. The resulting macroscopic model is no longer a standard Stokes system, but a Stokes system with memory. Remarkably, the character of this memory---oscillatory and persistent versus decaying---depends on a single dimensionless parameter that encodes the relative size of the inclusions. These phenomena, invisible in any single-phase model, emerge naturally through the mathematical framework of \emph{homogenization theory}.

\par\medskip

The homogenization of Stokes flow in perforated domains with \emph{rigid} obstacles is by now a classical subject. In a landmark series of papers, Allaire \cite{All1,All2} established that, depending on the relative size of the obstacles, the effective model can range from the Stokes equations with a Brinkman correction to those with a Darcy term. Closely related is the work of Cioranescu and Murat \cite{CiMu}, who identified the appearance of ``strange terms'' in the homogenization of elliptic problems in perforated domains. When the obstacles are rigid, the energy stored in the interaction is entirely dissipative, and the effective models remain local in time.

\par\medskip

A fundamentally different situation arises when the inclusions are \emph{elastic}: they can deform, store potential energy, and vibrate. The interplay between the viscous dissipation of the surrounding fluid and the elastic oscillations of the inclusions creates a mechanism for \emph{memory effects} in the effective medium. The emergence of non-local terms in time from homogenization of systems that are local in time was first identified by Sanchez-Palencia \cite{San}, who showed that periodic mixtures of elastic solids and viscous fluids lead to viscoelastic effective media with exponentially decaying memory kernels. Unlike the setting in \cite{San}, where both phases are connected, our work considers isolated elastic inclusions. Francfort and Suquet \cite{FrSu} demonstrated that composite materials with instantaneous viscous dissipation at the microscale can homogenize into effective media with fading memory, showing how microscopic friction transforms into macroscopic long-range relaxation. In a related direction, Briane and Casado-D\'{\i}az \cite{BrCa,BrCa2} proved that the homogenization of magneto-elastodynamic systems with strongly varying magnetic fields and with soft inclusions gives rise to a viscoelastic effective behavior with memory terms, further illustrating how the elastic-viscous coupling at the microscale generates non-local-in-time effects in the macroscopic model. For further developments on the homogenization of viscoelastic models and the rigorous derivation of memory effects, we also refer the reader to \cite{AKPS, DMSa}.

\par\medskip

The problem addressed in the present paper is most closely related to recent works on the homogenization of fluid-structure interaction systems in periodic media.

\par\medskip

In \cite{Cas2}, Casado-Díaz studied a coupled linearized elastodynamic--Stokes system in a periodic domain where the \emph{fluid phase} consists of isolated inclusions surrounded by a connected elastic matrix. This geometry models, in a simplified way, biological media such as skin, where cells can be idealized as fluid-filled compartments embedded in an extracellular matrix (see, e.g., \cite{Mur} for related spatial models in mathematical biology, and \cite{CFLM} for recent advances in the homogenization of liquid-filled elastomers). The main result of \cite{Cas2} is the derivation of a macroscopic viscoelastic equation with a convolution memory kernel acting on the strain history. The memory kernel decays exponentially to zero, the exponential decay being a consequence of the viscous dissipation inside the fluid inclusions.

\par\medskip

In the present work we consider the complementary geometry. We study an incompressible viscous fluid filling a three-dimensional bounded domain $\Omega\subset\mathbb R^3$, in which a large number of tiny elastic inclusions are periodically distributed. Specifically, the inclusions occupy a family of sets of the form $\varepsilon k+r_\varepsilon\mathcal O$ with $k\in\mathbb Z^3$, where $\varepsilon$ is the period and $r_\varepsilon\ll\varepsilon$ is the size of each inclusion. Now it is the fluid phase that is connected and the elastic phase that is disconnected. This inversion of the topology has deep consequences for the character of the memory. We note that this setting was already studied by Fleury \cite{Fleury} in the context of acoustic wave propagation. However, under the standard scaling assumptions of \cite{Fleury} (densities and compressibilities of both phases of the same order), no resonance-driven memory effects arise and the effective model reduces to a Darcy-type law with a complex permeability (for other rigorous derivations of Darcy's law in non-periodic suspensions of solid particles, we also refer to \cite{CaCaLu}). The contrast with our work is precisely the highly contrasted scaling, which triggers a resonant regime leading to qualitatively different memory terms.

\par\medskip

The microscopic model couples the linearized elasticity system in the solid inclusions with the Stokes equations in the surrounding fluid, with the standard interface conditions of velocity continuity and stress equilibrium on the boundary of each inclusion.  However, the density and the elastic coefficients of the solid must be rescaled in order for the inclusions---whose total volume fraction tends to zero---to have a non-trivial effect in the limit.  Specifically, we impose the density of order $\varepsilon^3/r_\varepsilon^3$ and the elasticity tensor of order $\varepsilon^3/r_\varepsilon$.  This critical scaling ensures that the kinetic energy and the elastic energy stored in each vanishingly small inclusion remain of the same order as the kinetic energy of the fluid.

\par\medskip

While this scaling may appear exotic at first sight, it is physically well-motivated by the theory of locally resonant acoustic metamaterials. In the pioneering work of Liu et al.\ \cite{LiuScience}, it was shown experimentally that composite materials consisting of heavy metallic cores coated by a thin layer of soft silicone rubber and embedded in an epoxy matrix can exhibit sonic band gaps at unexpected frequencies. The physical mechanism is local resonance: the heavy core, with its large inertia, vibrates against the restoring force of the soft coating, creating internal oscillators that interact with the propagating wave. From the theoretical point of view, Auriault and Boutin \cite{AuBo} used multi-scale asymptotic homogenization to show rigorously that in highly contrasted elastic composites---including three-component systems with hard inclusions coated by soft material---long-wavelength macroscopic waves induce local resonances in the soft phase, generating a frequency dependent effective tensorial density different from the real scalar density. At the macroscopic scale, non local effect in time appears. The key feature of these metamaterials, from our perspective, is precisely that the resonant units are \emph{heavy} (high density) and enclosed in \emph{soft} elastic coatings (low effective stiffness), which matches our scaling regime. In this way, our critical regime can be understood as a continuum-mechanics idealization of locally resonant metamaterial inclusions: the parameter $\varepsilon^3/r_\varepsilon^3$ represents the high effective density of the composite inclusion, while $\varepsilon^3/r_\varepsilon$ represents its low effective stiffness.

\par\medskip

Our main results are the following. After establishing the well-posedness and energy estimates for the $\varepsilon$-dependent system, we pass to the limit using the periodic unfolding method and the techniques developed in \cite{Cas}. The asymptotic behavior depends crucially on the limit parameter $\gamma = \lim_{\varepsilon\to 0} r_\varepsilon / \varepsilon^3 \in [0, +\infty]$. We extend our study to all possible values of $\gamma$, the most interesting cases being $\gamma \in (0, +\infty)$ and $\gamma = \infty$.

For $\gamma = 0$, the solid particles are too small to affect the macroscopic flow, and the fluid and solid equations uncouple in the limit. However, for $\gamma \in (0, +\infty]$, we show that the limit behavior can be described by a single closed integro-differential equation for the macroscopic fluid velocity. This equation takes the form of a \emph{Stokes system with memory}: the standard Stokes equations are modified by a time-convolution term of the velocity.

For the particular case $\gamma = \infty$, the macroscopic velocity $v$ of the fluid and the microscopic deformation $\hat u$ of the inclusions satisfy a coupled two-scale system  with the fluid velocity as Dirichlet boundary data for the elastic deformation. For this case $\gamma = \infty$, the explicit form of the convolution kernels can be obtained via spectral decomposition of the Dirichlet elastic operator, illustrating an oscillatory behavior of the memory kernel, which differs fundamentally from the exponentially decaying kernels obtained in the model with a connected solid phase \cite{Cas2}. Since the elastic inclusions are isolated and the microscopic equation has no internal viscous dissipation, the microscopic elastic oscillations are purely conservative. As a result, the memory kernel is \emph{oscillatory and almost-periodic}---it does not decay at infinity. This represents a qualitatively new type of effective behavior: a Newtonian viscous fluid whose macroscopic response acquires a persistent, non-decaying memory due to the purely conservative elastic vibrations of the suspended particles.

In contrast, for the intermediate regime $\gamma \in (0, +\infty)$, the microscopic elastodynamic equation is coupled to a local microscopic Stokes problem. The viscous dissipation of the local Stokes problem damps the elastic oscillations, and we show that the memory kernels decay to zero at infinity (Theorem \ref{enlgafi}  and Remark~\ref{rem:damped}). Thus, in contrast with the persistent, non-decaying memory of the case $\gamma=\infty$, the effective memory is of fading type.

\par\medskip

\subsection*{Organization of the paper}

The paper is organized as follows. In Section~2, we collect the notation used throughout the paper. In Section~3, we introduce the microscopic model and the scaling assumptions. Section~4 establishes uniform a priori estimates for the deformation, velocity, and pressure. In Section~5, we use these estimates together with the periodic unfolding method to pass to the limit and derive the homogenized system, first as a coupled two-scale problem and then as a closed integro-differential equation with memory for the macroscopic fluid velocity.

\section{Notation}
\begin{itemize}
    \item We denote by $\R^{3\times 3}_s$ the space of symmetric matrices of dimension $3\times 3$ with real coefficients.
    \item For two vectors $a,b\in\R^3$, we write $a\cdot b=\sum_{i=1}^3 a_ib_i$. For two matrices $\xi,\eta\in\R^{3\times 3}$, we write $\xi:\eta={\rm trace}(\xi\eta^T)=\sum_{i,j=1}^3\xi_{ij}\eta_{ij}$.
    \item For a vector function $u$, we denote by $e(u)$ the symmetric part of its gradient, i.e., $e(u)=(Du+Du^T)/2$. If there is an ambiguity on the variable of derivation, we use $e_x(u)$ or $e_y(u)$.
    \item For a measurable set $U\subset\R^3$, we denote by $\chi_U$ its characteristic function. For a measurable function $u$ in $U$, we denote by $u\chi_U$ its extension by zero to $\R^3$.
    \item For two normed spaces $X_1$, $X_2$, we denote by ${\cal L}(X_1;X_2)$ the space of continuous linear functions from $X_1$ into $X_2$. When $X_1=X_2$, we simply write ${\cal L}(X_1)$.
    \item For $g\in H^{-1/2}(\Gamma)$ and $u\in H^{1/2}(\Gamma)$ we use the notation $\int_\Gamma gu\,d\sigma$ to mean the duality product $\langle g,u\rangle_{H^{-1/2}(\Gamma),H^{1/2}(\Gamma)}$.
    \item For an open set $U\subset\R^3$, we denote by ${\cal D}(U)$ the space of $C^\infty$-functions with compact support in $U$, and by ${\cal D}'(U)$ its topological dual, i.e., the space of distributions in $U$.
    \item $D^{1,2}(\R^3\setminus\overline{{\cal O}})$ denotes the space of functions $w\in L^6(\R^3\setminus\overline{{\cal O}})$ such that $\nabla w\in L^2(\R^3\setminus\overline{{\cal O}})$.
    \item We denote by $I$ the identity operator. The underlying space will be understood from the context.
    \item We denote by $C$ a generic positive constant which can change from line to line.
    \item We denote by $O_\ep$ a generic sequence which tends to zero with $\ep$ and can change from line to line.
    \item We denote by $Y$ the unit cube $(-1/2,1/2)^3\subset\R^3$.
    \item For $k\in\ZZ^3$ and $\ep>0$, we denote $C_k^\ep:=\ep(k+Y)$.
    \item We define $\kappa:\R^3\to\ZZ^3$ by the condition $\kappa(x)\in\ZZ^3$, $x-\kappa(x)\in Y$, a.e.\ $x\in\R^3$. For $\ep>0$, $\ep\kappa(x/\ep)$ is the center of the cube $C_k^\ep$ containing $x$.
\end{itemize}

\section{Problem setting}
Our goal is to study the asymptotic behavior of an incompressible viscous fluid in which small elastic particles of size $r_\eps$ are immersed, arranged periodically with a small period $\eps\gg r_\eps$. Assuming that the density and the viscosity of the fluid are constants of order one, we impose the density of the solid particles of order $\eps^3/r_\eps^3$ and the coefficients of the corresponding elasticity tensor of order $\eps^3/r_\eps$. Note that this implies that the density is of a much higher order of magnitude than the elastic coefficients. Thus, the elastic particles are highly deformable. \par
Assuming that the elastic deformations are sufficiently small, the problem can be modeled by the linearized system
\begin{equation}
    \begin{cases}\dis
        {\eps^3\over r_\eps^3}\rho^s\partial^2_{tt} u_{\eps} - {\eps^3\over r_\eps} \divg_x(A e_x(u_{\eps})) = f_{\eps}             & \dis\text{in } (0,T)\times \Omega_{\eps}^s,       \\ \ecart
        \dis \rho^l \partial_t v_{\eps} - 2 \mu \divg_x(e_x(v_{\eps})) + \nabla_x p_{\eps} = f_{\eps}                               & \dis \text{in } (0,T) \times \Omega_{\eps}^l,     \\ \ecart
        \divg_x v_{\eps} = 0                                                                                                        & \text{in } (0,T) \times \Omega_{\eps}^l,          \\ \ecart
        \dis    \partial_t u_{\eps} = v_{\eps}, \ {\eps^3\over r_\eps} A e_x(u_{\eps}) \nu = (2 \mu e_x(v_{\eps}) - p_{\eps} I) \nu & \text{on } (0,T) \times \partial \Omega_{\eps}^s, \\ \ecart
        v_{\eps} = 0                                                                                                                & \text{on } (0,T) \times \partial \Omega,          \\ \ecart
        \left. u_{\eps} \right|_{t=0} = u_{\eps}^0, \ \left. \partial_t u_{\eps} \right|_{t=0} = u_{\eps}^1                         & \text{in } \Omega_{\eps}^s,                       \\ \ecart
        \left. v_{\eps} \right|_{t=0} = v_{\eps}^0                                                                                  & \text{in } \Omega_{\eps}^l,
    \end{cases}
    \label{Peps}
\end{equation}
where $\Om\subset\R^3$ is a connected bounded open set  whose boundary  can be locally written as the graph of a Lipschitz function
and $\eps,r_\eps>0$, such that
\begin{equation}\label{relpar}
    \eps\to 0,\qquad  \frac{r_{\eps}}{\eps} \to 0.
\end{equation}
The solid part $\Om_\eps^s$ and the fluid part $\Om_\eps^l$ are defined as
\begin{equation}\label{deOme}
    \Omega_{\eps}^s = \bigcup_{\eps k+\eps Y\subset\Om\atop k\in \ZZ^3}( \eps k + r_{\eps}{\cal O}), \qquad \Omega_{\eps}^l = \Omega \setminus \overline{\Omega_{\eps}^s},\end{equation}
where
\beq\label{defYO} {\cal O}\ \hbox{ open, with }\ \overline{{\cal O}}\subset Y:=\Big(-{1\over 2},{1\over 2}\Big)^3.
\eeq
The boundary of the open set ${\cal O}$  is assumed to be   locally  the graph of a Lipschitz function. Moreover $Y\setminus \overline{{\cal O}}$ is connected. \par
Through translation we can also assume that the center of mass of $\partial{\cal O}$ coincides with the origin, that is,
\begin{equation} \label{prcmO}
    \int_{\partial{\cal O}} y \, d\sigma(y) = 0.
\end{equation}\par
We will also introduce
\beq\label{defOmse} \Om^\ast_\ep=\bigcup_{\ep (k+Y)\subset\Om}\ep(k+Y).\eeq\par
The coefficients $\rho^s,\rho^l$ and $\mu$ in (\ref{Peps}) are assumed to be positive constants,  the tensor $A$ belongs to the space $\mathcal{L}\left(\R^{3\times 3}_s\right)$ and satisfies the ellipticity and symmetry conditions
\begin{equation} \label{proA}
    \left\{\ba{l} \exists\, \alpha > 0 \text{ such that } A\xi : \xi \geq \alpha |\xi|^2 \quad \forall \xi \in \R^{3\times 3}_s,\\ \ecart\dis
    A \xi : \eta = A \eta : \xi \quad \forall \xi, \eta \in \R^{3\times 3}_s.\ea\right.
\end{equation}
The vector $\nu$ in (\ref{Peps}) refers to the outer unit normal vector on $\partial\Om_\ep^s$ that point outwards ${\cal O}$.\par
With respect to the right-hand side and the initial conditions let us assume
\beq\label{hipcoe} \ba{l}\dis f_{\eps} \in L^1(0,T;L^2(\Omega))^3,\quad u_{\eps}^0 \in H^1(\Omega_{\eps}^s)^3,\quad u_{\eps}^1 \in L^2(\Omega_{\eps}^s)^3,\\ \ecart\dis v_{\eps}^0 \in L^2(\Omega_{\eps}^l)^3,\quad {\rm div}_xv_{\eps}^0=0\ \hbox{ in }\Om_\ep^l,\quad v_{\eps}^0 \cdot\nu=0\ \hbox{ on }\partial\Om,\ea\eeq
where now $\nu$ is  the outer unit normal vector on $\partial\Om$.\par
From the physical point of view, the unknown functions $u_\eps$, $v_\eps$ and $p_\eps$ respectively represent the deformation of the solid particles, the velocity of the fluid and its pressure.  The coefficients $\eps^3 \rho^s/r_\eps^3$, $\rho^l$ and $\mu$ are the density of the elastic particles, the density of the fluid and the viscosity coefficient. The tensor  $\eps^3 A/r_\eps$ is the elasticity tensor and $f_\eps$ is the exterior force acting on $\Om$.\par\medskip
The following theorem proves the existence of a unique solution for (\ref{Peps}). It does not use the particular structure of our set $\Om_\eps^s$.
A similar result where, contrary to the case presented here, it is the solid part which surrounds the liquid part has been proved in \cite{Cas2}.

\par

\begin{theorem}[Existence and uniqueness] \label{Thexsol} Let $\omega, \Omega$ be two bounded open sets in $\R^3$ such that $\overline{\omega} \subset \Omega$, whose boundaries can be locally written as the graph of a Lipschitz function. We also consider three positive constants $\rho^s,\rho^l,T$, and  a tensor $A\in {\cal L}(\R^{3\times 3}_s)$ which satisfies (\ref{proA}).
    Then, for every
    $f \in L^1(0,T;L^2(\Omega))^3$, $u^0 \in H^1(\omega)^3$, $u^1 \in L^2(\omega)^3$ and $v^0 \in L^2(\Omega \setminus \overline{\omega})^3$ such that
    $${\rm div}_x v^0=0\ \hbox{ in }\Om\setminus\overline\om,\quad v^0\cdot\nu=0\ \hbox{ on }\partial\om,$$
    there exists a unique solution $(u, v, p)$ of
    \begin{equation}
        \begin{cases}
            \rho^s \partial_{tt} u - {\rm div}_x(A e_x(u)) = f                      & \text{in } (0,T) \times \omega,                               \\ \ecart
            \rho^l \partial_t v - 2\mu\, {\rm div}_x\,e_x(v) + \nabla_x p = f       & \text{in } (0,T) \times (\Omega \setminus \overline{\omega}), \\ \ecart
            {\rm div}_x v = 0                                                       & \text{in } (0,T) \times (\Omega \setminus \overline{\omega}), \\ \ecart
            \partial_t u = v, \ A e_x(u) \nu = (2\mu e_x(v) - pI) \nu               & \text{on } (0,T) \times \partial \omega,                      \\ \ecart
            v = 0                                                                   & \text{on } (0,T) \times \partial \Omega,                      \\ \ecart
            \left. u \right|_{t=0} = u^0, \ \left. \partial_t u \right|_{t=0} = u^1 & \text{in } \omega,                                            \\ \ecart
            \left. v \right|_{t=0} = v^0                                            & \text{in } \Omega \setminus \overline{\omega},
        \end{cases}
        \label{FSI}
    \end{equation}
    with
    $$
        u \in C^1([0,T];L^2(\om))^3\cap C^0([0,T];H^1(\om))^3,
    $$
    $$        v \in C^0([0,T]; L^2(\Om\setminus\overline\om))^3,\quad        p={dP\over dt}\ \hbox{ with }P  \in C^0([0,T];L^2(\Omega \setminus \overline{\omega})),\ P_{|t=0}=0\ \hbox{ in }\Om\setminus\overline\om.
    $$
    Moreover, such solution satisfies the energy identity
    \beq\begin{array}{l}\dis
        \frac{1}{2} \frac{d}{dt} \left( \rho^s \int_{\omega} |\partial_t u|^2 \, dx + \int_{\omega} A e_x(u) : e_x(u) \, dx + \rho^l \int_{\Omega \setminus \overline{\omega}} |v|^2 \, dx \right) \\ \ecart\dis\hskip2cm
        + 2\mu \int_{\Omega \setminus \overline{\omega}} |e_x(v)|^2 \, dx = \int_{\omega} f \cdot \partial_t u \, dx + \int_{\Omega \setminus \overline{\omega}} f \cdot v \, dx\quad \hbox{in }(0,T).
        \label{EnergyIdentityFSI}
    \end{array}\eeq
\end{theorem}

\begin{proof}
    As in \cite{Cas2}, Theorem 3.2, let us use the Lumer-Philips theorem which is equivalent to the Hille--Yosida theorem (see e.g. \cite{Paz}, theorem 4.3),  Introducing $w = \partial_t u$, we rewrite \eqref{FSI} as the first-order system
    \beq\label{PbgFE}\left\{\ba{ll}\dis \partial_t u=w&\hbox{ in }(0,T)\times \om,\\ \ecart\dis \partial_t  w={1\over \rho^s}{\rm div}_x\big(Ae_x(u)\big)+{1\over \rho^s}f &\dis \hbox{ in }(0,T)\times\om,\\ \ecart\dis \partial_t v={2\mu\over\rho^l}{\rm div}_xe_x(v)-{1\over\rho^l}\nabla_xp+{1\over\rho^l}f&\hbox{ in }(0,T)\times(\Om\setminus\overline\om),\\ \ecart\dis {\rm div}_xv=0&\hbox{ in }(0,T)\times(\Om\setminus\overline\om),\\ \ecart\dis w=v,\ Ae_x(u)\nu=\big(2\mu e_x(v)-pI\big)\nu&\hbox{ on }(0,T)\times\partial\om,\\ \ecart\dis v=0&\hbox{ on }(0,T)\times\partial\Om,\\ \ecart\dis u_{|t=0}=u^0,\ w_{|t=0}=u^1&\hbox{ in }\om,\\ \ecart\dis v_{|t=0}=v^0&\hbox{ in }\Om\setminus\overline\om.\ea\right.\eeq
    For
    $$H_v:=\big\{v\in L^2(\Om\setminus\overline \om)^3:\ {\rm div}_xv=0\ \hbox{ in }\Om\setminus\overline\om,\ \ v\cdot\nu=0\ \hbox{ on }\partial\Om\big\},$$
    we introduce the Hilbert space
    $${\mathcal H}=H^1(\om)^3\times L^2(\om)^3\times H_v,$$
    endowed with the scalar product
    \beq\label{ecuteSD0}\ba{l}\dis \big((u_1,w_1,v_1),(u_2,w_2,v_2)\big)=\int_\om\big( Ae_x(u_1):e_x(u_2)+\rho^su_1\cdot u_2+\rho^s w_1\cdot w_2\big)dx\\ \ecart\dis \qquad\qquad+\rho^l\int_{\Om\setminus\overline\om}v_1\cdot v_2\,dx,\qquad \forall(u_1,w_1,v_1),(u_2,w_2,v_2)\in {\mathcal H}.\ea$$
        In this space we define the unbounded operator ${\mathcal A}:D({\mathcal A})\subset {\cal H}\to {\cal H}$ by
    $$\ba{l}\dis D({\mathcal A})=\Big\{(u,w,v)\in {\cal H}: \ {\rm div}_x\big(Ae_x(u)\big)\in L^2(\om)^3 ,\ w\in H^1(\om)^3, \\ \ecart\dis\qquad\qquad\quad\exists p\in L^2(\Om\setminus\overline\om)\ \hbox{ such that } 2\mu\,{\rm div}_xe_x(v)+\nabla_x p\in H_v\\ \ecart\dis \qquad\qquad \qquad\ w=v\hbox{ on }\partial\om,\
    Ae_x(u)\nu-2\mu e_x(v)\nu+p\nu=0\ \hbox{ on }\partial\om\Big\}, \ea\eeq
    and
    $${\mathcal A}(u,w,v)=\Big(w,{1\over \rho^s}{\rm div}_x\big(Ae_x(u)\big), {2\mu\over\rho^l}{\rm div}_xe_x(v)-{1\over\rho^l}\nabla_xp\Big).$$
    The operator  ${\mathcal A}$ is well defined since the function  $p$ in the definition of $D({\mathcal A})$ is unique. Namely, if $p_1, p_2$ satisfy such  conditions, then  $p_1-p_2$ solves
    $$\left\{\ba{l}\dis \Delta(p_1-p_2)=0\ \hbox{ in }\Om\setminus\overline\om\\ \ecart\dis p_1-p_2=0\ \hbox{ on }\partial\om,\quad \nabla (p_1-p_2)\cdot \nu=0\ \hbox{ on }\partial\Om,\ea\right.$$
    and thus $p_1=p_2$ in $\Om\setminus\overline\om.$\par
    In order to prove that ${\mathcal A}-I$ is in the conditions of the Lumer-Philips theorem and then that  ${\mathcal A}$ defines a semigroup, we have to prove that ${\mathcal A}-I$  dissipative and surjective. \par
    That ${\mathcal A}-I$  is dissipative, just follows by an integration by parts, taking into account the definition of
    ${\mathcal A}$. Namely, we have
    \beq\label{mothex}\ba{l}\dis \big(({\mathcal A}-I)(u,w,v),(u,w,v),\big)_{\mathcal H}\\ \ecart\dis=
    \int_\om \big(Ae_x(w):e_x(u)+\rho^s w\cdot u+{\rm div}_x(Ae_x(u))\cdot w\big)dx\\ \ecart\dis+\int_{\Om\setminus\overline\om}\big(2\mu\,{\rm div}_xe_x(v)-\nabla_xp\big)\cdot v\,dx-\|(u,w,v)\|^2_{\mathcal H}\\ \ecart\dis
    =\rho^s\int_\om w\cdot u\,dx- 2 \mu \int_{\Om\setminus\overline\om}|e_x(v)|^2dx-\|(u,w,v)\|_{\mathcal H}^2\leq -{1\over 2}\|(u,w,v)\|^2_{\mathcal H}.\ea\eeq
    For the surjectivity,  given $(f,g,h)\in {\mathcal H}$, we have to prove the existence of $(u,w,v)\in D({\mathcal A})$ such that
    \beq\label{sissobTE}({\mathcal A}-I)(u,w,v)=(f,g,h).\eeq
    This  is a system of three equations, the first of them reads as  $w-u=f.$ Eliminating $u$ from this equation and recalling the definition of $D({\mathcal A})$, we deduce that (\ref{sissobTE}) is equivalent to find $(w,v,p)\in H^1(\om)^3\times H^1(\Om\setminus\overline\om)^3\times L^2(\Om\setminus\overline\om)$ such that
    \beq\label{ecuteSD2}\left\{\ba{ll}\dis -{\rm div}_x(Ae_x(w))+\rho^sw=- \rho^sg-{\rm div}_x(Ae_x(f))&\dis \hbox{ in }\om,\\ \ecart\dis -2\mu\,{\rm div}_xe_x(v)+\nabla_xp+\rho^l v=-\rho^l h&\hbox{ in }\Om\setminus\overline\om,\\ \ecart\dis {\rm div}_xv=0&\hbox{ in }\Om\setminus\overline\om,\\ \ecart\dis w=v,\ Ae_x(w-f)\nu=\big(2\mu e_x(v)-pI\big)\nu&\hbox{ on }\partial\om,\\ \ecart\dis v=0&\hbox{ on }\partial\Om.\ea\right.\eeq
    Let us apply the Lax-Milgram theorem to prove the existence of solution for (\ref{ecuteSD2}). We define
    $$\widehat {\mathcal H}=\Big\{(w,v)\in H^1(\om)^3\cap H^1(\Om\setminus \overline\om)^3: \ w=v\hbox{ on }\partial\om,\ v=0\hbox{ on }\partial\Om\Big\},$$
    $$\widetilde {\mathcal H}=\Big\{(w,v)\in \widehat {\mathcal H}: \ {\rm div}_xv=0\hbox{ in }\Om\setminus\overline\om\Big\},$$
    and {\Large${\frak a}$}:$\,\widehat {\mathcal H}\times\widehat {\mathcal H}\to \RR$ the continuous bilinear form given by
    $$\ba{l}\dis\hbox{\Large${\frak a}$}\big((w_1,v_1),(w_2,v_2)\big)=\int_\om \big(Ae_x(w_1):e_x(w_2)+\rho^s w_1\cdot w_2\big)dx\\ \ecart\dis
        +\int_{\Om\setminus\overline\om} \big(2\mu\,e_x(v_1):e_x(v_2)+\rho^l v_1\cdot v_2\big)dx,\qquad \forall\, (w_1,v_1),(w_2,v_2)\in \widehat{\mathcal H}.\ea$$
    Reasoning as in (\ref{mothex}) we have that {\Large${\frak a}$} is coercive and thus, by the Lax-Milgram theorem there exists a unique solution of
    $$\left\{\ba{l}\dis (w,v)\in \widetilde {\mathcal H}\\ \ecart\dis
        \hbox{\Large${\frak a}$}\big((w,v),(\widetilde w,\widetilde v)\big)=\int_\om \big(-\rho^s g\cdot \widetilde w+Ae_x(f):e_x(\widetilde w)\big)dx-\rho^l\int_{\Om\setminus\overline\om}h\cdot\widetilde v\,dx,\ \forall\,(\widetilde w,\widetilde v)\in \widetilde {\mathcal H}.\ea\right.$$
    Then, defining $E:\widehat {\mathcal H}\to \RR$ by
    $$E(\widehat w,\widehat v)=\hbox{\Large${\frak a}$}\big((w,v),(\widehat w,\widehat v)\big)-\int_\om \big(-\rho^s g\cdot \widehat w+Ae_x(f):e_x(\widehat w)\big)dx+\rho^l\int_{\Om\setminus\overline\om}h\cdot\widehat v\,dx,$$
    we have that $E(\widehat w,\widehat v)=0$ for every $(\widehat w,\widehat v)\in \widehat {\mathcal H}$ with ${\rm div}_xv=0$ in $\Om\setminus\overline\om.$
    In particular this holds for $\widehat w=0$, $\widehat v\in H^1_0(\Om\setminus\overline\om)^3$ with null divergence. This allows us to use De Rham's theorem to deduce the existence of  $p\in L^2(\Om\setminus\overline\om)$, unique, up to an additive constant such that
    $$E(0,\widehat v)=- \int_{\Om\setminus\overline\om}p\,{\rm div}_x \widehat v\,dx,\ \ \forall\,\widehat v\in H^1_0(\Om\setminus\overline\om)^3.$$
    Since for every $\widehat w\in H^1(\om)^3$ with
    \beq\label{prowhw}\int_{\partial\om}\widehat w\cdot\nu\,ds(x)=0,\eeq
    there exists
    $\widehat v\in H^1_0(\Om\setminus\overline\om)^3$ with null divergence in $\Om\setminus\overline\om$ such that $\widehat v=\widehat w$ on $\partial\om$, we get that $w,v,p$ satisfy
    \beq\label{exiSD1}\ba{l}\dis\int_\om \big(Ae_x(w):e_x(\widehat w)+\rho^s w\cdot \widehat w\big)dx+\int_{\Om\setminus\overline\om} \big(2\mu\,e_x(v):e_x(\widehat v)+\rho^l v\cdot \widehat v-p\,{\rm div}_x\widehat v\big)dx\\ \ecart\dis
    =\int_\om \big(-\rho^s g\cdot \widehat w+Ae_x(f):e_x(\widehat w)\big)dx-\rho^l\int_{\Om\setminus\overline\om}h\cdot\widehat v\,dx,\ea\eeq
    for every $(\widehat w,\widehat v)\in \widehat {\mathcal H}$ which satisfies (\ref{prowhw}). Therefore $(\widehat w,\widehat v)\in \widehat {\mathcal H}$ is solution of a problem similar to (\ref{ecuteSD2}) that arises from replacing the second condition in the fourth line
    by
    $$Ae_x(w-f)\nu-\big(2\mu e_x(v)-pI\big)\nu=c\nu\ \hbox{ on }\partial\om,$$
    with $c$ constant.
    Since $p$ was defined up to an additive constant, we can choose it in such way that  $c=0$. This proves the existence of a unique solution for (\ref{ecuteSD2}). \par
    The Lumer-Philips theorem then proves the existence of a continuous semigroup $S(t):{\cal H}\to {\cal H}$  such that for every $(u^0,u^1,v^0)\in {\mathcal H}$, and every $f\in L^1(0,T;L^2(\Om))^3$ with
    \beq\label{confasTE}{\rm div}_xf=0\ \hbox{ in }(0,T)\times (\Om\setminus\overline \om),\quad f\cdot\nu=0\ \hbox{ on }(0,T)\times\partial\Om,\eeq
    i.e. such that $(0,f\chi_\om,f\chi_{\Om\setminus\overline\om})\in L^1(0,T;{\cal H})$,
    there exist a unique solution $(u,v,w,p)$ of (\ref{PbgFE}), where $(u,v,w)\in L^\infty(0,T;{\mathcal H})$ and it is given by
    $$(u(t),w((t),v(t))=S(t)(u^0,w^0,v^0)+\int_0^tS(t-r)\Big(0,{1\over\rho^s}f(r),{1\over \rho^l}f(r)\Big)dr.$$
    The general case follows by noting that every $f\in L^1(0,T:L^2(\Om))^3$
    can be decomposed as $f_1+f_2$ where $f_1$ satisfies (\ref{confasTE}) and $f_2$ is such that
    $$f_2=0\ \hbox{ in }\omega,\quad \int_{\Om\setminus\overline\om} f_2\cdot v\,dx=0,\ \forall\, v\in H_v.$$
    De Rham's Theorem then proves the existence of a distribution  $p\in L^1(0,T;L^2(\om))$  such that
    $$\nabla_xp=f\ \hbox{ in }(0,T)\times (\Om\setminus\overline\om),\quad p=0\ \hbox{ on }(0,T)\times \partial\om.$$
    A solution of (\ref{PbgFE}) with $f=f_2$ and null initial data is given for such  $p$ and
    $u,v,w$ the null functions. By linearity we then have the existence of solution for (\ref{PbgFE}). \par
    The semigroup theory also provides the energy identity (\ref{EnergyIdentityFSI}), which formally can be obtained by multiplying the first equation in (\ref{FSI}) by $\partial_tu$, the second one by $v$ and then integrating by parts.\par
    The fact that $p$ can be written as the time derivative of a function in $C^0([0,T];L^2(\om))$ vanishing at $t=0$ can be deduced  integrating in time in (\ref{FSI}) and then using De Rham's theorem. We will use a similar reasoning in the proof of Proposition \ref{prop:apriori} to get some a priori estimates for the solution of (\ref{Peps}).
\end{proof}

\section{A priori estimates}
Thanks to Theorem \ref{Thexsol}  we know that \eqref{Peps} has a unique solution for every $\ep>0$. Our aim in this section is to provide some a priori estimates for such solution which will be used in the next one to carry out the homogenization. This is given by the following proposition (see also Remark \ref{Re2PriEst} below)\par

\begin{proposition}
    \label{prop:apriori} Let $(u_\eps,v_\eps,p_\eps)$ be the solution of (\ref{Peps}), where $\Om_\ep^s,\Om_\ep^l$ are defined by (\ref{deOme}), $\rho^s,\rho^l,\mu$ are positive, $A\in {\cal L}(\RR^{3\times 3}_s)$ satisfies (\ref{proA}) and the right-hand side $f_\ep$ and the initial conditions $u_\ep^0,$ $u_\ep^1$, $v_\ep^0$ satisfy (\ref{hipcoe}). Then, denoting $D_\ep$ by
    $$D_\ep(t)=   \frac{\eps^3}{r_{\eps}^3}\|u_{\eps}^1\|^2_{L^2(\Omega_{\eps}^s)^3} +
        {\eps^3\over r_\eps}\|e_x(u_{\eps}^0)\|^2_{L^2(\Omega_{\eps}^s)^{3\times 3}} + \|v^0_{\eps}\|^2_{L^2(\Omega_{\eps}^l)^3}+\left(\int_0^t\|f_\eps\|_{L^2(\Om_\eps^l)^3}ds\right)^2,$$
    we have:\par
    There exists $C>0$, which only depends on $\rho^s$, $\rho^l$, $\alpha$ and $|A|$ such that
    \begin{equation}\ba{l}\dis
        \frac{\eps^3}{r_{\eps}^3}\|\partial_t u_{\eps}(t)\|^2_{L^2(\Omega_{\eps}^s)^3}\hskip-2pt + \hskip-2pt
        {\eps^{3}\over r_\eps}\|e_x(u_{\eps})(t)\|^2_{L^2(\Omega_{\eps}^s)^{3\times 3}}\hskip-2pt  +\hskip-2pt  \|v_{\eps}(t)\|^2_{L^2(\Omega_{\eps}^l)^3}\hskip-2pt\\ \ecart\dis
        +\hskip-4pt \int_0^t\hskip-4pt \|e_x(v_\eps)\|^2_{L^2(\Om_\eps^l))^{3\times 3}}ds\leq
        CD_\ep(t),\ \ \forall\,t\in [0,T].\ea
        \label{Bounduv}
    \end{equation}
    Moreover, there exist  $P_{\eps} \in C^0([0,T];L^2(\Omega))$, $\zeta_\ep\in C^0([0,T])$
    such that
    \beq\label{proPeps} p_\ep=\partial_t P_\ep\ \hbox{ in }(0,T)\times \Om_\ep^l,\eeq
    \beq\label{proPeps2}  P_\ep(t,x)=\zeta_\ep(t)\ \hbox{ a.e. in }\Om_\ep^s,\eeq
    \beq\label{proPeps3}
    P_\ep(0,x)=0\ \hbox{ a.e. in }\Om,\quad \into P_\ep(t,x)dx=0,\ \forall\, t\in [0,T],\eeq
    \beq\label{EcvPe}\ba{l}\dis \int_{\Om} P_\eps{\rm div}_x\varphi\,dx
        ={\eps^3\over r_\eps^3}\rho^s\int_{\Om_\eps^s}\big(\partial_t u_\eps-u_\eps^1\big)\cdot\varphi\,dx+\rho^l\int_{\Om_\eps^l}\big(v_\eps-v_\eps^0\big)\cdot\varphi\,dx\\ \ecart\dis +{\eps^3\over r_\eps}\int_{\Om_\eps^s}Ae_x\Big(\int_0^tu_\eps\,ds\Big):e_x(\varphi)dx
    +2\mu\int_{\Om_\eps^l}e_x\Big(\int_0^tv_\eps\,ds\Big):e_x(\varphi)dx\\ \ecart\dis-\into\int_0^tf_\eps\,ds\cdot\varphi\,dx,\quad \forall\,\varphi\in H^1_0(\Om)^3,\ \forall\,t\in [0,T].
    \ea\eeq
    There exists $C>0$ which only depends on $\Om$, ${\cal O}$, $\rho^s$, $\rho^l$, $\alpha$, and $|A|$ such that
    \beq\label{proPeps4}
    \|P_\ep(t)\|_{L^2(\Om)}\leq C\Big(1+{\ep^3\over r_\ep}\Big)^{1\over 2}D_\ep(t)^{1\over 2},\ \ |\zeta_\ep(t)|\leq
    C\Big[{\ep^3\over r_\ep^3}\Big(1+{\ep^3\over r_\ep}\Big)\Big]^{1\over 2}D_\ep(t)^{1\over 2},\ \  \forall\, t\in [0,T].
    \eeq
\end{proposition}

\begin{proof} Defining $E_\eps:[0,T]\to \R$ by
    \begin{equation*}
        E_{\eps}(t) := \rho_{\eps}^s\int_{\Omega_{\eps}^s}  |\partial_t u_{\eps}|^2 \, dx + \int_{\Omega_{\eps}^s} A_{\eps} e_x(u_{\eps}) : e_x(u_{\eps}) \, dx + \rho^l \int_{\Omega_{\eps}^l} |v_{\eps}|^2 \, dx,\ \ \forall\,t\in [0,T],
    \end{equation*}
    we can use the energy identity (\ref{EnergyIdentityFSI}), applied to (\ref{Peps}), and the Cauchy-Schwarz inequality to deduce
    \begin{equation}\label{ecuEP1}
        \frac{1}{2} \frac{dE_{\eps}}{dt} +2\mu\int_{\Om_\eps^l}|e_x(v_\eps)|^2dx\leq \|f_{\eps}\|_{L^2(\Omega)^3} \sqrt{E_{\eps}} \quad \text{ a.e. in } (0,T),
    \end{equation}
    which implies
    \begin{equation*}
        \sqrt{E_{\eps}(t)} \leq \sqrt{E_{\eps}(0)} + \|f_{\eps}\|_{L^1(0,t;L^2(\Omega)^3)}, \ \hbox{ in }(0,T),
    \end{equation*}
    and then (\ref{Bounduv}).\par\medskip
    In order to estimate the pressure $p_\ep$, we take $\varphi\in H^1_0(\Om)^3$ with ${\rm div}_x \varphi=0$ in $\Om_\eps^l$, as test function in the first and second equation in (\ref{Peps}). Integrating in $(0,t)\times\Om$, with $t\in (0,T)$ and taking into account the second boundary condition in the fourth line, we deduce that $G_\eps\in C^0([0,T];H^{-1}(\Om))^3$, defined as
    \beq\label{ecuvarGe}\ba{l}\dis \langle G_\eps(t),\varphi\rangle_{H^{-1}(\Om)^3,H^1_0(\Om)^3}
    :={\eps^3\over r_\eps^3}\rho^s\int_{\Om_\eps^s}\big(\partial_t u_\eps(t)-u_\eps^1\big)\cdot\varphi\,dx\\ \ecart\dis
    +\rho^l\int_{\Om_\eps^l}\big(v_\eps(t)-v_\eps^0\big)\cdot\varphi\,dx +{\eps^3\over r_\eps}\int_{\Om_\eps^s}Ae_x\Big(\int_0^tu_\eps\,ds\Big):e_x(\varphi)dx
    \\ \ecart\dis +2\mu\int_{\Om_\eps^l}e_x\Big(\int_0^tv_\eps\,ds\Big):e_x(\varphi)dx-\into\int_0^tf_\eps\,ds\cdot\varphi\,dx,
    \ea\eeq
    for every $\varphi\in H^1_0(\Om)^3$, and every $t\in [0,T]$, satisfies
    \beq\label{EcuEP2}\langle G_\eps(t),\varphi\rangle_{H^{-1}(\Om)^3,H^1_0(\Om)^3}=0,   \quad\forall\, \varphi\in H^1_0(\Om)^3,\ \hbox{ with }{\rm div}_x\varphi=0\hbox{ in }\Om_\ep^l.
    \eeq
    Since in particular (\ref{EcuEP2}) holds true for every $\varphi\in H^1_0(\Om)^3$ with ${\rm div}_x\varphi=0$ in $\Om$, De Rham's theorem proves the existence of  $P_\ep\in C^0([0,T];L^2(\Om))$, with $P_\ep(0)=0$ such that
    \beq\label{EcuEP3}\langle G_\eps(t),\varphi\rangle_{H^{-1}(\Om)^3,H^1_0(\Om)^3}=\into P_\ep(t)\,{\rm div}_x\varphi\,dx,\quad \forall\,\varphi\in H^1_0(\Om)^3, \ \forall\,t\in [0,T],\eeq
    \beq\label{EcuEP3b}\into P_\ep(t)dx=0,\quad\big\|P_\ep(t)\big\|_{L^2(\Om)}\leq C\|G_\ep(t)\|_{H^{-1}(\Om)^3},\ \forall\,t\in[0,T].\eeq
    \par
    Let us prove that  (\ref{EcuEP2}) and (\ref{EcuEP3}) imply the existence of  $\zeta_\ep\in C^0([0,T])$, such that $P_\ep(t)=\zeta_\ep(t)$ a.e. in
    $\Om_\ep^s$, for every $t\in [0,T]$. For this purpose, we take $k_1,k_2\in \ZZ^3$, $k_1\not= k_2$ such that $\ep (k_1+Y), \ep(k_2+Y)\subset \Om,$ and we use that for every  $\psi_1\in L^2(\ep k_1+r_\ep{\cal O})$, $\psi_2\in L^2(\ep k_2+r_\ep{\cal O})$, with
    \beq\label{EcuEP3bb}\int_{\ep k_1+r_\ep {\cal O}}\hskip-5pt \psi_1dx+\int_{\ep k_2+r_\ep{\cal O}}\hskip-5pt\psi_2dx=0,\eeq
    there exists $\varphi\in H^1_0(\Om)^3$ such that
    $${\rm div}_x\varphi=\psi_1\hbox{ in }\ep k_1+r_\ep {\cal O},\quad{\rm div}_x\varphi=\psi_2\hbox{ in } \ep k_2+r_\ep {\cal O},\quad{\rm div}_x\varphi=0\ \hbox{ otherwise.}$$
    Then, by (\ref{EcuEP3})
    $$\int_{\ep k_1+r_\ep {\cal O}}\hskip-5pt P_\ep(t)\psi_1dx+\int_{\ep k_2+r_\ep {\cal O}}\hskip-5pt P_\ep(t)\psi_2dx=\langle G_\ep,\varphi\rangle_{H^{-1}(\Om)^3,H^1_0(\Om)^3}=0.$$
    Since this holds for every $\psi_1,$ $\psi_2$ satisfying (\ref{EcuEP3bb}) we deduce the existence of $\zeta_\ep$ such that
    $$P_\ep(t)=\zeta_\ep(t)\ \hbox{ in }(\ep k_1+r_\ep {\cal O})\cup (\ep k_2+r_\ep {\cal O}),\ \ \forall\,t\in [0,T].$$
    By the arbitrariness of $k_1,k_2$ we get (\ref{proPeps2}).\par
    In order to estimate $\|P_\ep(t)\|_{L^2(\Om)}$ let us estimate $\|G_\ep(t)\|_{H^{-1}(\Om)^3}$ and $|\zeta_\ep(t)|$. \par
    Since $G_\ep(t)$ is defined as the sum of several terms, we need to estimate the norm in $H^{-1}(\Om)^3$ of each one. The unique which presents  some difficulties is the first one. We use that for every $\varphi\in H^1_0(\Om)^3$, (\ref{Bounduv}) gives
    \beq\label{EcuEP3c}\ba{l}\dis {\ep^3\over r_\ep^3}\left| \int_{\Om_\eps^s}\big(\partial_t u_\eps(t)-u_\eps^1\big)\cdot\varphi\,dx\right|\leq C{\ep^{3/ 2}\over r_\ep^{3/2}}D_\ep(t)^{1\over 2} \left(\int_{\Om_\ep^s}\big|\varphi\big|^2dx\right)^{1\over 2}.\ea\eeq
    In the right-hand side we use that thanks to H\"older's and Sobolev's inequalities, for every $k\in \ZZ^3$, we have
    $$ \int_{\ep k+r_\ep{\cal O}}\Big|\varphi-{1\over \ep^3}\int_{C_k^\ep}\varphi\,dy\Big|^2dx\leq C\left( \int_{C_k^\ep}\Big|\varphi-{1\over \ep^3}\int_{C_k^\ep}\varphi\,dy\Big|^6dx\right)^{1\over 3}r_\ep^2 \leq Cr_\ep^2\int_{C_k^\ep}|D_x\varphi|^2dx,$$
    and then
    $$ \int_{\ep k+r_\ep{\cal O}}\hskip-5pt|\varphi|^2dx\leq C{r_\ep^3\over \ep^3}\int_{C_k^\ep}|\varphi|^2dx+Cr_\ep^2\int_{C_k^\ep}|D_x\varphi|^2dx.$$
    Adding in $k$, and replacing in (\ref{EcuEP3c}), we obtain
    $$\ba{l}\dis {\ep^3\over r_\ep^3}\left| \int_{\Om_\eps^s}\big(\partial_t u_\eps(t)-u_\eps^1\big)\cdot\varphi\,dx\right|\leq CD_\ep(t)^{1\over 2} \left(\into|\varphi|^2dx+C{\ep^3\over r_\ep}\into|D_x\varphi|^2dx\right)^{1\over 2}.\ea$$
    Using this estimate in the definition of $G_\ep$, and taking into account Poincar\'e's inequality, we deduce
    \beq\label{EcuEP3d}\|G_\ep(t)\|_{H^{-1}(\Om)^3}
    \leq C\Big(1+{\ep^3\over r_\ep}\Big)^{1\over 2}D_\ep(t)^{1\over 2},\quad \forall\, t\in [0,T],\eeq
    which combined with (\ref{EcuEP3b}) proves the estimate for $\|P_\ep(t)\|_{L^2(\Om)}$ in (\ref{proPeps4}).\par
    Let us now estimate $|\zeta_\ep(t)|$. For this purpose, we define  $\psi\in H^1(Y)^3$ by
    $$\psi=1\ \hbox{ in }Y\setminus {r_\ep\over\ep}\overline{{\cal O}},\qquad \psi=\Big(1-{\ep^3\over r_\ep^3|{\cal O}|}\Big)\ \hbox{ in }{r_\ep\over\ep}{\cal O}.$$
    Then, since $\psi$ has null integral in $Y$, there exists $\phi\in H^1_0(Y)^3$ such that
    $${\rm div}_x\phi=\psi\ \hbox{ in }Y,\qquad\|\phi\|_{H^1_0(Y)^3}\leq C\|\psi\|_{L^2(Y)}\leq C{\ep^{3/ 2}\over r_\ep^{3/ 2}}.$$
    From this function $\phi$ we define $\varphi$ as
    $$\varphi(x)=\left\{\ba{ll}\dis \ep \phi\Big({x-\ep k\over \ep}\Big)&\ \hbox{ if }\ep(k+Y)\subset\Om\\ \ecart\dis 0  &\hbox{ otherwise}\ea\right.$$
    Using it in (\ref{EcuEP3}) we get
    $$\Big(1-{\ep^3\over r_\ep^3|{\cal O}|}\Big)\zeta_\ep(t)|\Om_\ep^s|+\int_{\Om_\ep^l}P_\ep(t,x) dx=\langle G_\ep(t),\varphi\rangle_{H^{-1}(\Om)^3,H^1_0(\Om)^3},$$
    which recalling that $P_\ep(t)$ has null integral in $\Om$, and then that
    $$\int_{\Om_\ep^l}P_\ep(t,x) dx=-\int_{\Om_\ep^s}P_\ep(t,x) dx=-\zeta_\ep(t)|\Om_\ep^s|,$$
    provides
    $$|\zeta_\ep(t)|\leq
        C\Big[{\ep^3\over r_\ep^3}\Big(1+{\ep^3\over r_\ep}\Big)\Big]^{1\over 2}D_\ep(t)^{1\over 2}.$$\end{proof}
\begin{remark} \label{Re2PriEst} By (\ref{proPeps}) and (\ref{proPeps4}), we get the following  estimate for the pressure
    \beq\label{espeL2}\|p_\ep\|_{W^{-1,\infty}(0,T; L^2(\Om_\ep^l))}\leq C\Big(1+{\ep^3\over r_\ep}\Big)^{1\over 2}\|D_\varepsilon\|_{L^\infty(0,T)}^{1\over 2}.\eeq
    Another estimate can be obtained using the results in \cite{All2} where it is proved that if $L_\ep\in H^{-1}(\Om_\ep^l)^3$ satisfies
    $$\langle L_\ep,\varphi\rangle_{H^{-1}(\Om_\ep^l)^3,H^1_0(\Om_\ep^l)^3}=0,\quad \forall\, \varphi\in H^1_0(\Om_\ep^l)^3\hbox{ with }{\rm div}_x\varphi=0\hbox{ in }\Om_\ep^l,$$
    then there exists $q_\ep\in L^2(\Om)$ such that
    $$\langle L_\ep,\varphi\rangle_{H^{-1}(\Om_\ep^l)^3,H^1_0(\Om_\ep^l)^3}=\int_{\Om_\ep^l}q_\ep\,{\rm div}_x\varphi\,dx,\quad \forall\varphi\in H^1_0(\Om_\ep^l)^3,$$
    $$ \|q_\ep\|_{L^2(\Om)}\leq C\Big(1+{r_\ep\over \ep^3}\Big)^{1\over 2}\|L_\ep\|_{H^{-1}(\Om_\ep^l)^3},\quad \into q_\ep\,dx=0.$$
    Applying this result to the distribution $G_\ep(t)$ in the proof of Proposition \ref{prop:apriori},  we deduce the existence of $Q_\ep\in C^0([0,T];L^2(\Om))$, with $Q_\ep(0)=0$, such that
    $$\langle G_\ep(t),\varphi\rangle_{H^{-1}(\Om_\ep^l)^3,H^1_0(\Om_\ep^l)^3}=\int_{\Om_\ep^l}Q_\ep(t,x){\rm div}_x\varphi\,dx,\ \ \forall\varphi\in H^1_0(\Om_\ep^l)^3,$$
    \beq\label{relPQ2} \into Q_\ep\,dx=0,\ \|Q_\ep\|_{L^2(\Om)}\leq C\Big(1+{r_\ep\over \ep^3}\Big)^{1\over 2}\|G_\ep\|_{H^{-1}(\Om_\ep^l)^3}\ \hbox{ a.e. in }(0,T),\eeq
    where using the definition of $G_\ep$ in the proof of Proposition \ref{prop:apriori} it is simple to check
    $$\|G_\ep(t)\|_{H^{-1}(\Om_\ep^l)^3}\leq CD_\ep(t)^{1\over 2},\quad\forall\,t\in [0,T].$$
    Comparing with (\ref{EcuEP3}) we deduce the existence of $\eta_\ep\in C^0([0,T])$, such that the function $P_\ep$ in Proposition \ref{prop:apriori} also satisfies
    \beq\label{relPQ3} P_\ep(t,x)=Q_\ep(t,x)+\eta_\ep(t)\ \hbox{ a.e. in }\Om_\ep^l,\ \ \forall t\in [0,T].\eeq
    Combining (\ref{proPeps}), (\ref{proPeps4}), (\ref{relPQ2}) and (\ref{relPQ3}) we deduce in particular the estimate
    \beq\label{Esgrapre}\|\nabla_x p_\ep\|_{W^{-1,\infty}(0,T; H^{-1}(\Om_\ep^l))^3}\leq C\|D_\varepsilon\|_{L^\infty(0,T)}^{1\over 2},\eeq
    where $C$ only depends on $\Om$, ${\cal O}$, $\rho^s$, $\rho^l$, $\alpha$, and $|A|$.
\end{remark}

\begin{remark} \label{EDPPe}
    Taking in (\ref{ecuvarGe}) $\varphi=\varphi(t,x)$ with $\varphi\in L^\infty(0,T; H^1_0(\Om))^3$, using (\ref{EcuEP3}) and Fubini's theorem, we get
    $$\ba{l}\dis \int_0^T\into P_\ep\, {\rm div}_x\varphi\,dxdt
        =-{\eps^3\over r_\eps^3}\rho^s\int_{\Om_\eps^s}u^1_\ep\cdot\int_0^T\varphi\,dtdx\\ \ecart\dis
        +{\eps^3\over r_\eps^3}\rho^s\int_0^T\int_{\Om_\eps^s}\partial_t u_\ep\cdot \varphi\,dxdt
        -\rho^l\int_{\Om_\eps^l}v^0_\ep\cdot\int_0^T\varphi\,dtdx\\ \ecart\dis+ \rho^l\int_0^T\int_{\Om_\eps^l}v_\ep\cdot \varphi\,dxdt
        +{\eps^3\over r_\eps}\int_0^T\int_{\Om_\eps^s}Ae_x(u_\ep):e_x\Big(\int_t^T\varphi\, ds\Big)dxdt
        \\ \ecart\dis +2\mu\int_0^T\int_{\Om_\eps^l}e_x(v_\eps):e_x\Big(\int_t^T\varphi\, ds\Big)dxdt-\int_0^T\into f_\eps \cdot\int_t^T\varphi\,ds\,dxdt, \ea$$
    that by rearranging the terms and renaming $\varphi$ to the function $\int_t^T\varphi\,ds$, in such way that now $\varphi$ is a function in $W^{1,\infty}(0,T;H^1_0(\Om))^3$ satisfying $\varphi(T,x)=0$ a.e. $x\in \Om$, proves that the solution of   the system (\ref{Peps}) satisfies the variational equation
    \beq\label{ecuvaSol}\ba{c}
    \ba{l}\dis
    -{\eps^3\over r_\eps^3}\rho^s\int_0^T\int_{\Om_\eps^s}\partial_t u_\ep\cdot\partial_t\varphi\,dxdt +{\eps^3\over r_\eps}\int_0^T\int_{\Om_\eps^s}Ae_x(u_\ep):e_x(\varphi)\,dxdt
    \\ \ecart\dis- \rho^l\int_0^T\int_{\Om_\eps^l}v_\ep\cdot\partial_t\varphi\,dxdt
    +2\mu\int_0^T\int_{\Om_\eps^l}e_x(v_\eps):e_x(\varphi)dxdt\\ \ecart\dis
    +\int_0^T\into P_\ep\, {\rm div}_x\partial_t\varphi\,dxdt=\int_0^T\into f_\eps \cdot\varphi\,dxdt\\ \ecart\dis
    + {\eps^3\over r_\eps^3}\rho^s\int_{\Om_\eps^s}u^1_\ep\cdot\varphi(0,x)\,dx+\rho^l\int_{\Om_\eps^l}v^0_\ep\cdot\varphi(0,x)\,dx,\\ \ecart\dis
    \forall\, \varphi\in W^{1,\infty}(0,T;H^1_0(\Om))^3,\ \hbox{ such that }\ \varphi(T,x)=0\ \hbox{ a.e. in }\Om.\ea\ea\eeq
\end{remark}
\section{The homogenization results}
The objective of this section is to carry out the homogenization of (\ref{Peps}). To this end we will adapt the ideas in \cite{Cas}, which are closely related to the unfolding and two-scale convergence methods (\cite{All}, \cite{ArDoHo}, \cite{CiDaGr}, \cite{Ngu}). Our main results are given by Theorems \ref{thm:co2e}, \ref{TeoCo} and \ref{enlgafi}.
\par
Taking into account Proposition \ref{prop:apriori}, let us assume that the initial data and the right-hand side $f_\ep$ in (\ref{FSI}) satisfy
\beq\label{acodatos} \left\{\hskip-3pt\ba{l}\dis \exists C>0\hbox{ such that for every }\ep>0,\hbox{ we have}\\ \ecart\dis {\eps^3\over r_\eps}\|u_{\eps}^0\|^2_{H^1(\Omega_{\eps}^s)^3}+
\frac{\eps^3}{r_{\eps}^3}\|u_{\eps}^1\|^2_{L^2(\Omega_{\eps}^s)^3}  + \|v^0_{\eps}\|^2_{L^2(\Omega_{\eps}^l)^3}+\|f_\eps\|^2_{L^1(0,T;L^2(\Om))^{3}}\leq C,\ea\right.
\eeq
Moreover, to avoid a limit force given by a measure, we will assume that $f_\ep$ satisfies the  equi-integrability condition
\beq\label{equiinfe} \forall\, \gamma>0,\ \ \exists\delta>0\ \hbox{ such that }
\int_E\|f_\ep\|_{L^2(\Om)^3}\leq \gamma,\ \ \forall\,\ep>0,\ \forall\, E\subset [0,T],\ \ |E|<\delta.\eeq
\par In these conditions Proposition \ref{prop:apriori} and Remark \ref{Re2PriEst} prove the existence of $C>0$ such that
\beq\label{espso}
\ba{l}\dis
\frac{\eps^3}{r_{\eps}^3}\|\partial_tu_{\eps}\|^2_{L^\infty(0,T;L^2(\Omega_{\eps}^s))^3}\hskip-2pt + \hskip-2pt
{\eps^{3}\over r_\eps}\|u_{\eps}\|^2_{L^\infty(0,T;H^1(\Omega_{\eps}^s))^{3\times 3}}\hskip-2pt  +\hskip-2pt  \|v_{\eps}\|^2_{L^\infty(0,T;L^2(\Omega_{\eps}^l))^3}\hskip-2pt\\ \ecart\dis
+\|e_x(v_\eps)\|^2_{L^2(0,T;L^2(\Om_\eps^l))^{3\times 3}}+\Big(1+{r_\ep\over\ep^3}\Big)\|p_\ep\|^2_{W^{-1,\infty}(0,T;L^2(\Om))}\\ \ecart\dis \hskip6cm+\|\nabla_xp_\ep\|_{W^{-1,\infty}(0,T;H^{-1}(\Om))^3}^2\leq
C.\ea
\eeq\par
As usual, we want to use these estimates to extract a subsequence for which the solutions $u_\ep$, $v_\ep$, $p_\ep$ converge to some limits that we want to identify. The first difficulty arises from the fact that these sequences are defined in spaces that vary with $\ep$. In this sense, we will use  the following lemma which is proved in  Theorems 4.2 and 4.4 in \cite{OlShYo}
\begin{lemma}
    There exists a linear extension operator $\mathcal{E}_{\eps}: H^1_{\partial \Omega}(\Omega_{\eps}^l)^3 \to H^1_0(\Omega)^3$ such that
    \begin{equation*}
        \mathcal{E}_{\eps} v = v \text{ in } \Omega_{\eps}^l, \quad \forall v \in H^1_{\partial \Omega}(\Omega_{\eps}^l)^3,
    \end{equation*}
    and
    \begin{equation}
        \|\mathcal{E}_{\eps} v\|_{H^1_0(\Omega)^3} \leq C \left(\|v\|_{L^2(\Omega_{\eps}^l)^3} + \|e_x(v)\|_{L^2(\Omega_{\eps}^l)^{3\times 3}}\right), \quad \forall v \in H^1_{\partial \Omega}(\Omega_{\eps}^l)^3,
        \label{Extensionu}
    \end{equation}
    with $C$ independent of $\eps$ and $v$.
\end{lemma}
Thanks to this lemma, we can extend $v_{\eps}^0$ to the whole of $\Omega$ and thus, using (\ref{acodatos}), assume that $v_{\eps}^0$ is a bounded sequence in $H^1_0(\Omega)^3$. Analogously, we can extend $v_\ep$ to $\Om$ to get a sequence which is bounded in
$W^{1,\infty}(0,T;L^2(\Om))^3\cap L^\infty(0,T;H^1_0(\Om))^3$. Now the question is what can we do with $u^0_\ep$, $u^1_\ep$ and $u_\ep$. It is here where we can apply the unfolding method (\cite{ArDoHo}, \cite{Cas}, \cite{CiDaGr}), which
will be also useful to study the behavior of $v_\ep$ near the interface $\partial\Om_\ep^s$. The idea is as follows: \par
In order to study the behavior of a sequence of measurable functions $w_\ep$ defined in the varying spaces $ \Om_\ep^s$, let us study the behavior of  the sequence $\hat w_\ep:\Om\times {\cal O}\to \RR$ defined as
$$\hat w_\ep(x,y)=w_\ep\Big(\ep\kappa\Big({x\over \ep}\Big)+r_\ep y\Big),\quad \hbox{ a.e. }(x,y)\in \Om^\ast_\ep\times{\cal O}.$$
Note that for $x\in C_k^\ep\subset\Om$, the sequence $\hat w_\ep(x,y)$ does not depend on $x$ while as function of $y$  it is just the result of using the change of variables
$$y={x-\ep k\over r_\ep},$$
which transforms the set $\ep k+r_\ep{\cal O}$ into ${\cal O}$. The variable $x$ represents the macroscopic variable and $y$ the microscopic one. \par
Analogously, we can use the same idea to study the behavior near $\Om_\ep^l$ of a sequence $w_\ep$ defined in the whole of $\Om$. \par
We will use the following compactness lemma
\begin{lemma} \label{lemcomuf}
    Let $u_\ep,v_\ep$ be two functions in $W^{1,\infty}(0,T;L^2(\Om_\ep^s))^3\cap L^{\infty}(0,T;H^1(\Om_\ep^s))^3$ and  $W^{1,\infty}(0,T;L^2(\Om))^3\cap L^2(0,T;H^1_0(\Om))^3$ respectively, such that
    \beq\label{hilecp2}\ba{l}\dis
    \frac{\eps^3}{r_{\eps}^3}\|\partial_tu_{\eps}\|^2_{L^\infty(0,T;L^2(\Omega_{\eps}^s))^3} +
    {\eps^{3}\over r_\eps}\|u_{\eps}\|^2_{L^\infty(0,T;H^1(\Omega_{\eps}^s))^3} \\ \ecart\dis +  \|v_{\eps}\|^2_{L^\infty(0,T;L^2(\Omega))^3}
    +\|D_xv_\eps\|^2_{L^2(0,T;L^2(\Om))^{3\times 3}}\leq
    C,\ea
    \eeq
    \beq\label{hidiv=0} {\rm div}_xv_\ep=0\ \hbox{ in }(0,T)\times\Om_\ep^l,\eeq
    \beq\label{hilecp3}\partial_tu_\ep=v_\ep\ \hbox{ on }(0,T)\times\partial\Om_\ep^s,
    \eeq
    then, defining $\hat u_\ep:(0,T)\times\Om^\ast_\ep\times {\cal O}\to\RR^3$ by
    \beq\label{defhue}  \hat u_\ep(t,x,y)=u_\ep\Big(t,\ep\kappa\Big({x\over\ep}\Big)+r_\ep y\Big)\ \hbox{ a.e. }(t,x,y)\in (0,T)\times\Om^\ast_\ep\times {\cal O}, \eeq
    there exist  two functions $\hat u\in W^{1,\infty}(0,T;L^2(\Om;L^2({\cal O})))^3\cap L^\infty(0,T;L^2(\Om;H^1({\cal O})))^3$ and $v\in W^{1,\infty}(0,T;L^2(\Om))^3\cap L^2(0,T;H^1_0(\Om))^3,$ with
    \beq\label{codiv0}{\rm div}_x v=0\ \hbox{ in }(0,T)\times\Om,\eeq
    \beq\label{confhu} {d\over dt}\int_{\partial{\cal O}} \hat u\cdot\nu\,ds(y)=0\ \hbox{ in }(0,T)\times \Om,\eeq
    such that  for a subsequence, we have
    \begin{equation}\label{codehue}\ba{c}\dis
        \hat u_{\eps}\stackrel{\ast} \rightharpoonup \hat u \quad \text{in } W^{1,\infty}(0,T;L^2(K;L^2({\cal O})))^3\cap L^\infty(0,T;L^2(K;H^1({\cal O})))^3\\ \ecart\dis\forall\, K\subset\Om,\ \hbox{ compact.}\ea
    \end{equation}
    \begin{equation}\label{codeve}
        v_{\eps} \rightharpoonup v \quad \text{in } L^2(0,T;H^1_0(\Omega))^3.
    \end{equation}
    Moreover, if we assume that
    \beq\label{hilcp1}\exists\lim_{\ep\to 0}{r_\ep\over\ep^3}=\gamma\in (0,\infty],\eeq
    then, defining $\hat v_\ep$ as ($v_\ep$ is assumed to be extended by zero outside $\Om$)
    \beq\label{defhve}\hat v_\ep(t,x,y)=v_\ep\Big(t,\ep\kappa\Big({x\over\ep}\Big)+r_\ep y\Big)\ \hbox{ a.e. }(t,x,y)\in (0,T)\times\Om\times (\ep Y/r_\ep),\eeq
    there exists $\hat v:(0,T)\times \Om\times \RR^3\to\RR^3$ such that, for a subsequence,
    we have
    \begin{equation}\label{codehve}
        \hat v_{\eps} \rightharpoonup \hat v \quad \text{in } L^2(0,T;L^2(\Omega;H^1(B(0,R)\setminus{\cal O})))^{3},\quad \forall R>0,
    \end{equation}
    \beq \label{dhuhv}\partial_t\hat u=\hat v\ \hbox{ on }(0,T)\times\Om\times\partial{\cal O}.\eeq
    \par\medskip
    If $\gamma=\infty$, then
    \beq\label{vainre}\hat v=v\ \hbox{ a.e. in }(0,T)\times\Om\times\RR^3.\eeq\par
    If $\gamma\in(0,\infty)$, then
    \beq\label{rehvgi} \hat v-v\in L^2(0,T;L^2(\Omega;D^{1,2}(\RR^3)))^{3}.\eeq
    \beq\label{diyhv0}{\rm div}_y\hat v=0\ \hbox{ in }(0,T)\times\Om\times\big(\RR^3\setminus{\cal O}\big)\eeq\par
\end{lemma}
\begin{remark} \label{reconfhu} For $\gamma\in (0,\infty]$ conditions (\ref{vainre}) or (\ref{rehvgi}), (\ref{diyhv0}) imply (\ref{confhu}). Thus, this condition only acts for $\gamma=0$.
\end{remark}
\begin{proof}[Proof of Lemma \ref{lemcomuf}] The existence of $v$ such that (\ref{codeve}) holds for a subsequence just follows from $v_\ep$ bounded in $W^{1,\infty}(0,T;L^2(\Om))^3\cap L^2(0,T;H^1_0(\Om))^3$ and having null divergence in $\Om_\ep^l$ where the measure of $\Om\setminus\Om_\ep^l$ goes to zero.\par
    Now, using the change of variables $y=(x-\ep k)/r_\ep$ in each cube $\ep(k+y)$, $k\in \ZZ^3$, we have
    $$
        \int_{\Omega_\ep^\ast} \int_{\cal O}\big( |\partial_t \hat{u}_{\eps}|^2 + |e_y(\hat{u}_{\eps})|^2 \big) dydx=     \frac{\eps^3}{r_{\eps}} \int_{\Omega_{\eps}^s} \Big({1\over r_\ep^2}|\partial_t u_{\eps}|^2 +
        |e_x(u_{\eps})|^2 \Big)dydx\leq C\ \hbox{ a.e. in }(0,T).
    $$
    This implies the existence of $\hat u\in W^{1,\infty}(0,T;L^2(\Om;L^2({\cal O})))^3\cap L^\infty(0,T;L^2(\Om;H^1({\cal O})))^3$ such that (\ref{codehue}) holds for a subsequence of $\ep$. \par
    Taking into account the boundary condition $\partial_t u_\ep=v_\ep$ on $(0,T)\times\partial\Om_\ep^s$, where $v_\ep$ vanishes on $\partial\Om$ and has null divergence in $\Om_\ep^l$, we have that
    $$\int_{\ep k+r_\ep\partial{\cal O}}\partial_t u_\ep\cdot\nu\,ds=0\ \hbox{ in }(0,T),\ \ \forall k\in \RR^3,\ k+\ep Y\subset\Om.$$
    As a consequence the function $\hat u$ must satisfy (\ref{confhu}).
    \par
    Now,  we introduce
    \begin{equation*}
        m_{\eps}(t,x) := \fint_{\ep Y/r_\ep}\hskip-8pt \hat{v}_{\eps}(t,x,y) \, dy=\fint_{\ep(\kappa({x\over\ep})+Y)}\hskip-10pt v_{\eps}(t,z) \, dz,\quad \hbox{ a.e. }(t,x)\in (0,T)\times\Om,
    \end{equation*}
    which can be easily seen to satisfy (see e.g. \cite{Cas})
    \beq\label{convme} m_\ep\rightharpoonup v\ \hbox{ in }L^2(0,T;L^2(\RR^3)).
    \eeq
    Then, for
    \beq\label{dehwe}\hat w_\ep=\hat v_\ep-m_\ep,\eeq
    we have
    \beq\label{eculc1}D_y\hat v_\ep=D_y\hat w_\ep\ \hbox{ a.e. in }(0,T)\times\RR^3\times (\ep Y/r_\ep),\eeq
    \beq\label{eculc2}{r_\ep\over \ep^3}\int_0^T\int_{\RR^3}\int_{{\ep\over r_\ep}Y} |D_y\hat w_\ep|^2dydxdt=\int_0^T\into|D_xv_\ep|^2dxdt\leq C,
    \eeq
    and then, by the Sobolev-Wirtinger inequality
    $${r_\ep\over \ep^3}\int_0^T\int_{\RR^3}\Big(\int_{{\ep\over r_\ep}Y} |\hat w_\ep|^6dy\Big)^{1\over 3}dxdt\leq {r_\ep\over \ep^3}\int_0^T\int_{\RR^3}\int_{{\ep\over r_\ep}Y} |D_y\hat w_\ep|^2dydxdt\leq C.
    $$
    By (\ref{hilecp3}) we also have
    $$\partial_t\hat u_\ep=\hat v_\ep \ \hbox{ in }(0,T)\times \Om\times \partial{\cal O}.$$
    Therefore, for a subsequence of $\ep$, we have:\par\medskip
    If $\gamma=\infty$, then
    $$\hat v_\ep\rightharpoonup v\ \hbox{ in }L^2(0,T; L^2(\Om;H^1(B_R)))^3,\ \ \forall R>0,\qquad \partial_t\hat u=v\ \hbox{ on }(0,T)\times \Om\times \partial{\cal O}.$$\par
    If $\gamma\in (0,\infty)$, then there exists $\hat v:(0,T)\times \Om\times\RR^3\to\RR^3$ such that
    \beq\label{conhv}\hat v_\ep\rightharpoonup \hat v\ \hbox{ in }L^2(0,T; L^2(\Om;H^1(B_R)))^3,\ \  \forall R>0,\quad \partial_t\hat u=\hat v\\ \hbox{ on }(0,T)\times \Om\times \partial{\cal O},\eeq
    \beq\label{hvmv}\hat v-v\in L^2(0,T; L^2(\Om;D^{1,2}(\RR^3))^3.\eeq
    Since ${\rm div}_y\hat v_\ep(t,x,y)=0$ for $y\not\in{\cal O}$, it is also clear that (\ref{diyhv0}) is satisfied.\end{proof}
\begin{theorem}
    \label{thm:co2e} Assume
    \beq\label{cove3r} \exists \lim_{\ep\to 0}{r_\ep\over\ep^3}=\gamma\in [0,\infty],
    \eeq
    and  that the right-hand sides and the initial data in (\ref{Peps}) are such that there exist $f\in L^1(0,T;L^2(\Om))^3$, $\hat u^0\in L^2(\Om;H^1({\cal O}))^3$, $\hat u^1\in L^2(\Om;L^2({\cal O}))^3$ and $v^0\in L^2(\Om)^3$ such that defining  $$\hat u^0_\ep(x,y)=u_\ep^0\Big(\ep\kappa\Big({x\over\ep}\Big)+r_\ep y\Big),\ \hat u^1_\ep(x,y)=u_\ep^1\Big(\ep\kappa\Big({x\over\ep}\Big)+r_\ep y\Big),\quad\hbox{a.e. }(x,y)\in \Om\times {\cal O},$$
    we have
    \beq\label{convfe} f_\ep\rightharpoonup f\ \hbox{ in }L^1(0,T; L^2(\Om))^3,\eeq
    \beq\label{convu0e} \hat u^0_\ep\rightharpoonup \hat u^0\ \hbox{ in }L^2(\Om;H^1({\cal O}))^3,\eeq
    \beq\label{convu1e} \hat u^1_\ep\rightharpoonup \hat u^1\ \hbox{ in }L^2(\Om;L^2({\cal O}))^3,\eeq
    \beq\label{convv0e} v^0_\ep\rightharpoonup v^0\ \hbox{ in }L^2(\Om)^3.\eeq
    Then, there exist $\hat u\in W^{1,\infty}(0,T;L^2(\Om;L^2({\cal O})))^3\cap L^\infty(0,T;L^2(\Om;H^1({\cal O}))^3$, with
    \beq\label{coninhu}\hat u(0,x,y)=\hat u^0(x,y),\ \hbox{ a.e. }(x,y)\in \Om\times{\cal O},\eeq
    $v\in W^{1,\infty}(0,T;L^2(\Om))^3\cap L^2(0,T;H^1_0(\Om))^3,$ and $p\in W^{-1,\infty}(0,T;L^2(\Om))$ such that the solutions $u_\ep$, $v_\ep$, $p_\ep$ of (\ref{Peps}) satisfy (\ref{codiv0}), (\ref{codehue}), (\ref{codeve}),  and
    \medskip
    \beq\label{convpr} \nabla_x p_\ep\rightharpoonup \nabla_xp\ \hbox{ in }W^{-1,\infty}(0,T;H^{-1}(\Om))^3,\eeq
    where $\hat u_\ep$ is defined by (\ref{defhue}).\par
    The functions $\hat u$, $v$, $p$ are characterized by ($p$ is defined up to a distribution in $W^{-1,\infty}(0,T)$):
    \par\medskip
    If $\gamma=0$,
    \beq\label{pblim2eg0a}\left\{\ba{ll}\dis \rho^s\partial^2_{tt}\hat u-{\rm div}_y(Ae_y(\hat u))=0&\dis \hbox{in }(0,T)\times \Om\times {\cal O},\\ \ecart\dis\partial_t\int_{\partial{\cal O}}\hat u\,ds(y)=0,\  Ae_y(\hat u)\nu=c(t,x)\nu&\dis\hbox{on }(0,T)\times\Om\times \partial{\cal O},\\ \ecart\dis \hat u_{|t=0}=\hat u^0,\ \partial_t\hat u_{|t=0}=\hat u^1&\dis \hbox{in }\Om\times {\cal O},\ea\right.\eeq
    \beq\label{pblim2eg0b}\left\{\ba{ll}\dis\rho^l\partial_t v-2\mu{\rm div}_xe_x(v)+\nabla_xp=f& \hbox{in }(0,T)\times\Om,\\ \ecart\dis
    {\rm div}_xv=0&\dis \hbox{in }(0,T)\times\Om,\\ \ecart\dis
    v=0&\dis \hbox{on }(0,T)\times\partial\Om,\\ \ecart\dis
    v_{|t=0}=v^0&\dis\hbox{in }\Om.\ea\right.\eeq
    \par\medskip
    If $\gamma=\infty$, then
    \beq\label{pblim2egin}\left\{\ba{ll}\dis \rho^s\partial^2_{tt}\int_{\cal O}\hat u\,dy+\rho^l\partial_tv-2\mu\,{\rm div}_x\big(e_x(v)\big)+\nabla_x p=f&\dis \hbox{in }(0,T)\times\Om,\\ \ecart\dis
    {\rm div}_xv=0&\dis\hbox{in }(0,T)\times\Om,\\ \ecart\dis
    \rho^s\partial^2_{tt}\hat u-{\rm div}_y(Ae_y(\hat u))=0&\hbox{in }(0,T)\times \Om\times {\cal O},\\ \ecart\dis
    \partial_t\hat u=v& \dis \hbox{on }(0,T)\times\Om\times \partial{\cal O}\\ \ecart\dis  v=0 &\dis\hbox{on }(0,T)\times\partial\Om,\\ \ecart\dis
    v_{|t=0}=v^0,\quad \hat u_{|t=0}=\hat u^0,\quad  \partial_t\hat u_{|t=0}=\hat u^1&\dis \hbox{in }\Om\times {\cal O}.\ea\right.\eeq\par
    If $\gamma\in (0,\infty)$, there also exist $\hat v$, with $\hat v-v\in L^2(0,T; L^2(\Om;D^{1,2}(\RR^3\setminus{\cal O})))^3$, and $\hat p\in W^{-1,\infty}(0,T;L^2(\Om;L^2(\RR^3\setminus{\cal O})))$ such that defining $\hat v_\ep$ by (\ref{defhve}) and $\hat p_\ep$ by
    \beq\label{defhpe}\hat p_\ep(t,x,y)=p_\ep\Big(t,\ep\kappa\big({x\over\ep}\big)+r_\ep y\Big)\ \hbox{ in }(0,T)\times\Om\times\big(\ep Y/r_\ep\big),\eeq
    we have (\ref{conhv}), (\ref{hvmv}) and
    \beq\label{conhpe} \nabla_y\hat p_\ep\rightharpoonup\nabla_y\hat p\ \hbox{ in }W^{-1,\infty}(0,T;L^2(\Om;L^2(B(0,R)\setminus{\cal O})))^3.\eeq
    with $\hat u,v,p,\hat v,\hat p$ defined by
    \beq\label{pblim2eg0in}\left\{\ba{ll}\dis \rho^s\partial^2_{tt}\int_{\cal O}\hat u\,dy+\rho^l\partial_tv-2\mu\,{\rm div}_x e_x(v)+\nabla_x p=f&\dis\hbox{in }(0,T)\times\Om,\\ \ecart\dis
    {\rm div}_xv=0&\dis \hbox{in }(0,T)\times\Om,\\ \ecart\dis
    \rho^s\partial^2_{tt}\hat u-{\rm div}_y(Ae_y(\hat u))=0 &\hbox{in }(0,T)\times \Om\times {\cal O},\\ \ecart\dis
    -2\mu\,{\rm div}_y e_y(\hat v)+\nabla_y\hat p=0 &\hbox{ in }(0,T)\times\Om\times (\RR^3\setminus{\cal O}),\\ \ecart\dis
    {\rm div}_y\hat v=0 &\hbox{in }(0,T)\times\Om\times (\RR^3\setminus{\cal O}),\\ \ecart\dis
    v=0&\dis\hbox{on }(0,T)\times\partial\Om,\\ \ecart\dis
    \partial_t\hat u=\hat v,\ Ae_y(\hat u)\nu=\gamma\big(2\mu e_y(\hat v)-\hat pI\big)\nu&\dis\hbox{on }(0,T)\times\Om\times \partial{\cal O},\\ \ecart\dis
    v_{|t=0}=v^0,\ \ \hat u_{|t=0}=u^0,\ \ \partial_t\hat u_{|t=0}=\hat u^1& \dis\hbox{in }\Om\times {\cal O}.\ea\right.\eeq
    \par\medskip
\end{theorem}
\begin{remark} Reasoning similarly to Lemma \ref{lemcomuf} it can be proved that (\ref{acodatos}) implies that, for a subsequence of $\ep$, there exist $\hat u^0\in L^2(\Om;H^1({\cal O}))^3$, $\hat u^1\in L^2(\Om;L^2({\cal O}))^3$ and $v^0\in L^2(\Om)^3$ such that (\ref{convu0e}), (\ref{convu1e}), (\ref{convv0e}) hold. Therefore, these assumptions are not restrictive.\par
    We also remark that due to ${\rm div}_xv^0_\ep=0$ in $\Om_\ep^l$,  we have
    \beq\label{divvo00} {\rm div}_xv^0=0\ \hbox{ in }\Om.\eeq
    \par
    The existence of $f\in L^1(0,T;L^2(\Om))^3$ satisfying (\ref{convfe}) for a subsequence of $\ep$,  also follows from (\ref{acodatos}) and (\ref{equiinfe}). \end{remark}

\begin{proof}[Proof of Theorem \ref{thm:co2e}] Thanks to (\ref{espso}), we can apply the compactness Lemma \ref{lemcomuf} to deduce the existence of  $\hat u$, $v$ and $\hat v$ the functions given by this lemma, which exist for a subsequence.\par
    By (\ref{codehue}), the function $\hat u^0$ satisfies
    \beq\label{cihu0} \hat u(0,x,y)=\hat u^0(x,y)\ \hbox{ a.e. in }\Om\times{\cal O}.\eeq
    \par\medskip
    Let us consider $\varphi\in C^1([0,T]; C^1_c(\Om))^3$ and $\phi\in C^1([0,T]\times \overline\Om\times \RR^3)^3$, such that
    \beq\label{prtho1} \varphi(T,x)=\phi(T,x,y)=0,\ \ \forall\,(x,y)\in \Om\times\RR^3,\eeq
    \beq\label{prtho3} \int_{\partial{\cal O}}\phi\cdot\nu\,ds(y)=0\ \hbox{ in } [0,T]\times \Om.
    \eeq
    \beq\label{prtho2}\left\{\ba{l}\dis\hbox{If }\gamma=\infty\Rightarrow \phi(t,x,y)=\varphi(t,x),\  \forall \,(t,x,y)\in [0,T]\times \Om\times (\RR^3\setminus{\cal O}), \\ \ecart\dis \hbox{If }\gamma<\infty\Rightarrow \exists\, R>0\ \hbox{ s.t. }\phi(t,x,y)=\varphi(t,x),\ \forall \,(t,x,y)\in [0,T]\times \Om\times \{|y|>R\}.\ea\right.
    \eeq
    For such $\varphi$ and $\phi$, we define $\Phi_\ep:[0,T]\times\Om\to\RR$ as
    $$\Phi_\ep(t,x) = \phi\Big(t,x,{x-\ep k\over r_\ep}\Big)\ \hbox{ a.e. in }(0,T)\times \ep(k+Y),\ \forall\,k\in\ZZ^3.$$
    Using $\Phi_\ep$ as test function in (\ref{ecuvaSol})  we get
    \beq \label{WeakForm} \ba{l}\dis        -{\ep^3\over r_\ep^3} \rho^s\int_0^T \int_{\Omega_{\eps}^s}  \partial_t u_{\eps} \cdot \partial_t \Phi_{\eps} \, dx dt +{\ep^3\over r_\ep} \int_0^T \int_{\Omega_{\eps}^s} A e_x(u_{\eps}) : e_x(\Phi_{\eps}) \, dx dt \\ \ecart\dis
    - \rho^l \int_0^T \int_{\Omega_{\eps}^l} v_{\eps} \cdot \partial_t \Phi_{\eps} \, dxdt + 2\mu \int_0^T \int_{\Omega_{\eps}^l} e_x(v_{\eps}) : e_x(\Phi_{\eps}) \, dxdt \\ \ecart\dis
    + \int_0^T \int_{\Omega} P_{\eps} \, \divg_x \partial_t\Phi_{\eps}\, dxdt
    = \int_0^T \int_{\Omega} f_{\eps} \cdot \Phi_{\eps} \, dxdt  \\ \ecart\dis
    +{\ep^3\over r_\ep^3}\rho^s \int_{\Omega_{\eps}^s}  u_{\eps}^1 \cdot \Phi_{\eps}(0,x) \, dx + \rho^l \int_{\Omega_{\eps}^l} v_{\eps}^0 \cdot \Phi_{\eps}(0,x) dx.\ea\eeq
    \par\medskip
    Let us  pass to the limit in each of the terms that appear in \eqref{WeakForm}.

    \par   \medskip\noindent
    {\it Terms 1 and 2.} Performing the change of variable $y = (x-\ep k)/r_\ep$ in each solid inclusion $\ep k+r_\ep {\cal O}$, $k\in\ZZ^3$,  and using (\ref{codehue}), we get
    \begin{equation}\label{t1g0}\ba{l}\dis
        {\ep^3\over r_\ep^3} \int_0^T\hskip-2pt \int_{\Omega_{\eps}^s}  \partial_t u_{\eps} \cdot \partial_t \Phi_{\eps} \, dx dt
        =\int_0^T\hskip-2pt  \into \int_{\cal O} \partial_t \hat{u}_{\eps}  \cdot \partial_t\phi\,dydxdt+O_\ep\\ \ecart\dis\to \int_0^T\hskip-2pt  \into \int_{\cal O} \partial_t \hat{u} \cdot\partial_t \phi\,dydxdt.\ea
    \end{equation}
    \beq\label{t2g0} \ba{l}\dis {\ep^3\over r_\ep} \int_0^T \int_{\Omega_{\eps}^s} A e_x(u_{\eps}) : e_x(\Phi_{\eps}) \, dx dt=\int_0^T \into\int_{\cal O} A e_y(\hat u_{\eps}) : e_y(\phi) \, dydx dt+O_\ep\\ \ecart\dis \to\int_0^T \into\int_{\cal O} A e_y(\hat u) : e_y(\phi) \, dydxdt.\ea\eeq\par\medskip\noindent
    {\it Term 3.} By (\ref{prtho2})  we have that $\Phi_\ep(t,x)=\varphi(t,x)$ except in a subset of $\Om$, independent of $t\in (0,T)$, whose measure tends to zero with $\ep$. Thus, using the Lebesgue dominated convergence theorem, we get
    \beq\label{coPhieco} \Phi_\ep\to \varphi\ \hbox{ in }H^1(0,T;L^2(\Om))^3.\eeq
    By (\ref{codeve}), we then have
    \beq\label{t3g0}  \int_0^T \int_{\Omega_{\eps}^l} v_{\eps} \cdot \partial_t \Phi_{\eps} \, dxdt\to \int_0^T \into v \cdot \partial_t \varphi \, dxdt.\eeq\par\medskip\noindent
    {\it Term 4.} If $\gamma=\infty$, then by (\ref{prtho2}), $\Phi_\ep=\varphi$ in $(0,T)\times \Om_\ep^l$ and thus, taking into account that $|\Om_\ep^s|\to 0$, we get
    \beq\label{t4g0}\int_0^T \int_{\Omega_{\eps}^l} e_x(v_{\eps}) : e_x(\Phi_{\eps}) \, dxdt\to \int_0^T\into e_x(v):e_x(\varphi)\,dxdt.\eeq
    For $\gamma<\infty$, taking into account (\ref{prtho2}), we also have for every $k\in \ZZ^3$
    \beq\label{expep}\ba{ll}\dis D_x\Phi_\ep &\dis=\Big(D_x\phi\Big(t,x,{x-\ep k\over r_\ep}\Big)+
    {1\over r_\ep}D_y\phi\Big(t,x,{x-\ep k\over r_\ep}\Big)\Big)\chi_{\{|x-\ep k|<Rr_\ep\}}\\ \ecart &\dis +D_x\varphi\chi_{\{|x-\ep k|>Rr_\ep\}}\ \hbox{ a.e. in }\ep(k+Y).\ea\eeq
    If $\gamma=0$, then
    \beq\label{estdpg0}{1\over r_\ep^2}\sum_{k\in\ZZ^2}\int_{C_k^\ep\cap\Om_\ep^l}|D_y\Phi\Big({x-\ep k\over r_\ep}\Big)\Big|^2dx\leq C{r_\ep\over \ep^3}\to 0,\eeq
    and therefore, similarly to (\ref{coPhieco}) we have
    $$D_x\Phi_\ep\to D_x\varphi\ \hbox{ in }L^2(0,T;L^2(\Om))^{3\times 3},$$
    which implies that (\ref{t4g0}) also holds true for $\gamma=0$.\par
    In the case $\gamma\in (0,\infty),$ taking into account (\ref{expep}), performing the change of variables $y=(x-\ep k)/r_\ep$ in each cube $\ep(k+Y)$, and using  (\ref{prtho2}), (\ref{codeve}) and (\ref{codehve}), we get
    \beq\label{t3ge0i}  \ba{l}\dis \int_0^T \int_{\Omega_{\eps}^l} e_x(v_{\eps}) : e_x(\Phi_{\eps}) \, dxdt\\ \ecart\dis =\sum_{k\in\ZZ^3}\int_0^T \int_{\Om_\ep^l\cap(\ep k+B_{R r_\ep}) } \hskip-30pt e_x(v_{\eps}): e_x(\Phi_\ep)\, dxdt+\sum_{k\in\ZZ^3}\int_0^T \int_{\Om_\ep^l\cap(C_k^\ep\setminus(\ep k+B_{R r_\ep}))} \hskip-30pt e_x(v_{\eps}): e_x(\Phi_\ep)\, dxdt\\ \ecart\dis ={r_\ep\over \ep^3}\int_0^T\int_{\Om\times (B_R\setminus {\cal O})}e_y(\hat v_\ep):e_y(\phi)\,dxdydt+\int_0^T\int_{\Om_\ep^l}e_x(v_\ep):e_x(\varphi)\,dxdt+O_\ep\\ \ecart\dis \to \gamma\int_0^T\int_{\Om\times (B_R\setminus {\cal O})}\hskip-8pt e_y(\hat v):e_y(\phi)\,dxdydt+\int_0^T\into e_x(v):e_x(\varphi)\,dxdt.\ea\eeq\par\medskip\noindent
    {\it Term 5.} If $\gamma=0$, using (\ref{proPeps2}), (\ref{relPQ3}) and (\ref{prtho3})  we have
    $$ \int_0^T \int_{\Omega} P_{\eps} \, \divg_x \partial_t\Phi_{\eps}\, dxdt =\int_0^T \int_{\Omega_\ep^l} Q_{\eps} \, \divg_x \partial_t\Phi_{\eps}\, dxdt
    $$
    where thanks to (\ref{relPQ2}), and $\gamma=0$, we can assume the existence of $Q\in L^\infty(0,T;L^2(\Om))$ with null integral in $\Om$ such that
    \beq\label{copreg0} Q_\ep\stackrel{\ast}\rightharpoonup Q\ \hbox{ in }L^\infty(0,T;L^2(\Om)).\eeq
    Using then (\ref{expep}) and (\ref{estdpg0}), applied to $\partial_t\Phi_\ep$ instead of $\Phi_\ep$, we get
    \beq\label{t5g0} \int_0^T \int_{\Omega} P_{\eps} \, \divg_x \partial_t\Phi_{\eps}\, dxdt \to \int_0^T \int_{\Omega} Q \, \divg_x \partial_t\varphi\, dxdt.
    \eeq\par
    If $\gamma>0$,  (\ref{proPeps2}) and (\ref{prtho3}) imply

    $$ \int_0^T \int_{\Omega} P_{\eps} \, \divg_x \partial_t\Phi_{\eps}\, dxdt =\int_0^T \int_{\Omega_\ep^l} P_{\eps} \, \divg_x \partial_t\Phi_{\eps}\, dxdt,
    $$
    where thanks to (\ref{proPeps4}), and extracting a subsequence if necessary, we can assume the existence of $P\in L^\infty(0,T;L^2(\Om))$ with null integral in $\Om$ such that
    \beq\label{conPe} P_\ep\stackrel{\ast}\rightharpoonup P\ \hbox{ in }L^\infty(0,T;L^2(\Om)).\eeq\par
    In the case $\gamma=\infty$, due to (\ref{prtho2}), we then have
    \beq\label{t5gin}\int_0^T \int_{\Omega} P_{\eps} \, \divg_x \partial_t\Phi_{\eps}\, dxdt\to \int_0^T \int_{\Omega} P \, \divg_x \partial_t\varphi\, dxdt.\eeq
    \par
    In the case $\gamma\in (0,\infty)$, taking into account that the sequence $\hat P_\ep$ defined as
    \beq\label{dehP} \hat P_\ep(t,x,y)=P_\ep\Big(t,\ep\kappa\Big({x\over\ep}\Big)+r_\ep y\Big)\ \hbox{ a.e. }(t,x,y)\in (0,T)\times\Om\times{\ep\over r_\ep}Y,\eeq
    satisfies
    $${r_\ep^3\over\ep^3}\int_{\Om\times Y/r_\ep}|\hat P_\ep|^2dxdy=\into |P_\ep|^2dx,\ \hbox{ a.e. in }(0,T),$$
    we can assume the existence of $\hat P\in L^\infty(0,T;L^2(\Om\times\RR^3))$ such that

    \beq\label{conhPe} r_\ep\hat P_\ep\chi_{\ep Y/r_\ep}\stackrel{\ast}\rightharpoonup \hat P\ \hbox{ in }L^\infty(0,T;L^2(\Om;L^2(\RR^3))).\eeq
    Then, reasoning similarly to the fourth term, we get
    \beq\label{t5ge0in}\ba{l}\dis \int_0^T \int_{\Omega} P_{\eps} \, \divg_x \partial_t\Phi_{\eps}\, dxdt\\ \ecart\dis \to\gamma\int_0^T \int_{\Omega\times ( B_R\setminus {\cal O})}\hskip-12pt \hat P \, \divg_y \partial_t\phi\, dxdydt+ \int_0^T \into P \, \divg_x \partial_t\varphi\, dxdt.\ea\eeq
    \par\medskip \noindent
    {\it Terms 6, 7, 8.} By (\ref{convfe}), (\ref{convu1e}) and (\ref{convv0e}), we have
    \beq\label{t6pl}  \int_0^T \int_{\Omega} f_{\eps} \cdot \Phi_{\eps} \, dxdt  \to \int_0^T\into f\cdot\varphi\,dxdt,\eeq
    \beq\label{t7pl}  {\ep^3\over r_\ep^3} \int_{\Omega_{\eps}^s}  u_{\eps}^1 \cdot \Phi_{\eps}(0,x) \, dx   \to \int_{\Om\times\cal O}\hat u^1\cdot\phi(0,x,y)\,dydx,\eeq
    \beq\label{t8pl}  \int_0^T \int_{\Omega} v^0_{\eps} \cdot \Phi_{\eps} (0,x)\, dxdt  \to  \into v^0\cdot\varphi(0,x)\,dx.\eeq
    \par\medskip
    Reasoning by density, we have then proved that the functions $\hat u,v$ and $\hat v$ satisfy the variational problems\par\medskip
    If $\gamma=0$
    \beq\label{pbvg0}\left\{\ba{l}\dis \hat u\in W^{1,\infty}(0,T;L^2(\Om;L^2({\cal O})))^3\cap L^\infty(0,T;L^2(\Om;H^1({\cal O})))^3, \\ \ecart\dis v\in L^\infty(0,T;L^2(\Om))^3\cap L^2(0,T;H^1_0(\Om))^3,\ Q\in L^\infty(0,T;L^2(\Om))\\ \ecart\dis {\rm div}_xv=0\hbox{ in }(0,T)\times\Om,\ \
    \hat u_{|t=0}=\hat u^0\ \hbox{ in }\Om\times {\cal O},\ \partial_t\hskip-2pt\int_{\partial {\cal O}}\hskip-6pt\hat u\cdot\nu ds(y)\hskip-2pt=\hskip-2pt 0\\ \ecart\dis
    -\rho^s\int_0^T\hskip-2pt  \into \int_{\cal O} \partial_t \hat{u} \cdot\partial_t \phi\,dydxdt+\int_0^T \into\int_{\cal O} A e_y(\hat u) : e_y(\phi) \, dydxdt\\ \ecart\dis -\rho^l  \int_0^T \into v \cdot \partial_t \varphi \, dxdt+2\mu  \int_0^T\into e_x(v):e_x(\varphi)\,dxdt \\ \ecart\dis+ \int_0^T \int_{\Omega} Q \, \divg_x \partial_t\varphi\, dxdt=\int_0^Tf\cdot\varphi\,dxdt\\ \ecart\dis
    +\rho^s\int_{\Om}\int_{\cal O}u^1\cdot\phi(0,x,y)\,dydx+\rho^l\into v^0\cdot\varphi(0,x)\,dx,\\ \ecart\dis
    \forall(\phi,\varphi)\in W^{1,1}(0,T; L^2(\Om;H^1({\cal O})))^3\times W^{1,1}(0,T; H^1(\Om))^3\\ \ecart\dis \phi(T,x,y)=\varphi(T,x)=0\hbox{ in }\Om\times{\cal O},\ \partial_t\hskip-3pt\int_{\partial {\cal O}}\hskip-6pt\phi\cdot\nu ds(y)\hskip-2pt=\hskip-2pt 0.\ea\right.\eeq
    If $\gamma=\infty$
    \beq\label{pbvgi}\left\{\ba{l}\dis \hat u\in W^{1,\infty}(0,T;L^2(\Om;L^2({\cal O})))^3\cap L^\infty(0,T;L^2(\Om;H^1({\cal O})))^3,\\ \ecart\dis v\in L^\infty(0,T;L^2(\Om))^3\cap L^2(0,T;H^1_0(\Om))^3,\ P\in L^\infty(0,T;L^2(\Om)),\\ \ecart\dis{\rm div}_xv=0\ \hbox{ in }(0,T)\times\Om,\\ \ecart\dis
    \partial_t\hat u=v\ \hbox{ on }(0,T)\times \Om\times\partial{\cal O}
    ,\ \hat u_{|t=0}=\hat u^0\ \hbox{ in }\Om\times {\cal O},\\ \ecart\dis
    -\rho^s\int_0^T\hskip-2pt  \into \int_{\cal O} \partial_t \hat{u} \cdot\partial_t \phi\,dydxdt+\int_0^T \into\int_{\cal O} A e_y(\hat u) : e_y(\phi) \, dydxdt\\ \ecart\dis -\rho^l  \int_0^T \into v \cdot \partial_t \varphi \, dxdt+2\mu  \int_0^T\into e_x(v):e_x(\varphi)\,dxdt \\ \ecart\dis+ \int_0^T \int_{\Omega} P \, \divg_x \partial_t\varphi\, dxdt = \int_0^T\into f\cdot\varphi\,dxdt \\ \ecart\dis
    +\rho^s\int_{\Om\times\cal O}u^1\cdot\phi(0,x,y)\,dydx+\rho^l\into v^0\cdot\varphi(0,x)\,dx,\\ \ecart\dis
    \forall\,\phi\in W^{1,1}(0,T; L^2(\Om;H^1({\cal O})))^3,\ \forall\,\varphi\in W^{1,1}(0,T; H^1(\Om))^3,\\ \ecart\dis
    \phi(t,x,y)=\varphi(t,x)\hbox{ on }(0,T)\times\Om\times \partial{\cal O},\ \int_{\partial{\cal O}}\hskip-3pt\phi\cdot\nu\,ds(y)=0 \hbox{ in }(0,T)\times\Om,
    \\ \ecart\dis \phi(T,x,y)=\varphi(T,x)=0\hbox{ in }\Om\times{\cal O}.\ea\right.\eeq
    If $\gamma\in(0,\infty)$
    \beq\label{pbvge0i}\left\{\ba{l}\dis\hat u\in W^{1,\infty}(0,T;L^2(\Om;L^2({\cal O})))^3\cap L^\infty(0,T;L^2(\Om;H^1({\cal O})))^3\\ \ecart\dis v\in L^\infty(0,T;L^2(\Om))^3\cap L^2(0,T;H^1_0(\Om))^3,\ \ P\in L^\infty(0,T;L^2(\Om)),\\ \ecart\dis\hat v\in L^2(0,T;L^2(\Om;H^1_{loc}(\RR^3\setminus {\cal O})))^3,\
    \hat P\in L^\infty(0,T;L^2(\Om;D^{1,2}(\RR^3\setminus{\cal O}))),\\ \ecart\dis
    {\rm div}_xv=0\hbox{ in }(0,T)\times\Om,\ {\rm div}_y\hat v=0\hbox{ in }(0,T)\times\Om\times (\RR^3\setminus{\cal O}),\\ \ecart\dis
    \hat u_{|t=0}=\hat u^0\ \hbox{ in }\Om\times {\cal O},\ \partial_t\hat u=v\ \hbox{ on }(0,T)\times \Om\times\partial{\cal O},\\ \ecart\dis
    \hat v-v\in L^2(0,T;L^2(\Om;D^{1,2}(\RR^3\setminus{\cal O}))^3,\ \partial_t\hat u=\hat v\ \hbox{ on }(0,T)\times\Om\times\partial{\cal O},\\ \ecart\dis -\rho^s\int_0^T\hskip-2pt  \into \int_{\cal O} \partial_t \hat{u} \cdot\partial_t \phi\,dydxdt+\int_0^T \into\int_{\cal O} A e_y(\hat u) : e_y(\phi) \, dydxdt\\ \ecart\dis -\rho^l  \int_0^T \into v \cdot \partial_t \varphi \, dxdt+2\mu  \int_0^T\into e_x(v):e_x(\varphi)\,dxdt \\ \ecart\dis+2\mu\gamma\int_0^T\int_{\Om\times(\RR^3\setminus{\cal O})}e_y(\hat v):e_y(\phi)\,dxdydt+ \int_0^T \int_{\Omega} P \, \divg_x \partial_t\varphi\, dxdt\\ \ecart\dis +\gamma\int_0^T\int_{\Om\times (\RR^3\setminus{\cal O})}\hat P{\rm div}_y\partial_t\phi\,dxdydt
    =\int_0^T\into f\cdot\varphi\,dxdt\\ \ecart\dis
    +\rho^s\int_{\Om\times\cal O}u^1\cdot\phi(0,x,y)\,dydx+\rho^l\into v^0\cdot\varphi(0,x)\,dx,\\ \ecart\dis
    \forall\,\phi\in W^{1,1}(0,T; L^2(\Om;H^1_{loc}(\RR^3))^3,\ \forall\,\varphi\in W^{1,1}(0,T; H^1(\Om))^3,\\ \ecart\dis
    \phi-\varphi\in W^{1,1}(0,T; L^2(\Om;D^{1,2}(\RR^3)))^3,\ \int_{\cal O}\phi \cdot\nu ds(y)=0,\\ \ecart\dis
    \phi(T,x,y)=\varphi(T,x)=0\hbox{ in } \Om\times {\cal O}.\ea\right.\eeq
    From these variational systems, taking $p=\partial_tP=\partial_t Q$, and recalling (\ref{divvo00}) we conclude the thesis of the theorem.\end{proof}
Theorem \ref{thm:co2e} provides some weak convergences for the solutions of (\ref{Peps}). The following theorem proves that these convergences hold in fact in the strong sense if the data converges strongly.
\begin{theorem}\label{TeoCo} Assume that the convergences in (\ref{convfe}), (\ref{convu0e}), (\ref{convu1e}) and (\ref{convv0e}) hold in the strong sense. Then in Theorem \ref{thm:co2e} we have
    \beq\label{cofhue}\int_{\Om_\ep^\ast}\int_{\cal O}\big(|\partial_t(\hat u_\ep-\hat u)|^2+|e_y(\hat u_\ep -\hat u)|^2\big)dydx+\int_{\Om_\ep^l}|v_\ep-v|^2dx\to 0\ \hbox{ in }L^\infty(0,T).\eeq
    For $	\gamma=0$ or $\gamma=\infty$, we also have
    \beq\label{confve}\int_0^T\int_{\Om_\ep^l}|e_x(v_\ep-v)|^2dxdt\to 0.\eeq
    In the case $\gamma\in (0,\infty)$, taking $R_\ep>0$ such that
    \beq\label{proRe}{R_\ep\over r_\ep}\to \infty,\quad {\ep\over R_\ep}\to 0,\eeq
    we have
    \beq\label{convhv} \int_0^T\int_{\Om\setminus B(\ep k,R_\ep)}\hskip-0.5cm|e_x(v_\ep-v)|^2dxds+\int_0^T \into\int_{B_{R_\ep/r_\ep}}\hskip-10pt\Big|{r_\ep\over\ep^3}e_y(\hat v_\ep) - \gamma e_y(\hat v)\Big|^2dydxds \to 0. \eeq
\end{theorem}
\begin{proof}  The result easily follows from the convergence of the energies. We just prove it in  the case $\gamma\in (0,\infty)$. Integrating in time (\ref{EnergyIdentityFSI}) applied to (\ref{Peps}), passing to the limit in the right-hand side thanks to the strong convergence of the right-hand side and the initial data, and using the energy identity for (\ref{pblim2eg0in}) we have
    $$\ba{l}\dis \int_{\Omega_\ep^s} \Big({\rho^s\ep^3\over 2r_\ep^3}|\partial_t u_\ep|^2 +{\ep^3\over 2r_\ep} A e_x(u_\ep) : e_x(u_\ep) \Big) dx + \int_{\Omega_\ep^l} \Big({\rho^l\over 2} |v_\ep|^2 + 2\mu \int_0^t|e_x(v_\ep)|^2ds\Big)dx \\ \ecart\dis
        = \int_0^t\Big(\int_{\Omega_\ep^s} f_\ep \cdot \partial_t u_\ep \, dx + \int_{\Omega_\ep^l} f_\ep \cdot v_\ep \, dx\Big)ds\\ \ecart\dis  +\int_{\Om_\ep^\ast} \int_{\cal O}\Big({\rho^s\over 2}|\hat u^1_\ep|^2+{1\over 2}A e_y(\hat u^0_\ep) : e_y(\hat u^0_\ep)\Big)dydx + {\rho^l\over 2} \int_{\Omega_\ep^l} |v^0_\ep|^2dx \\ \ecart\dis
        \longrightarrow \into \Big(\int_0^tf \cdot v \,ds + \int_{\cal O}\Big({\rho^s\over 2}|\hat u^1|^2+A e_y(\hat u^0) : e_y(\hat u^0)\Big) dy +{ \rho^l\over 2} \into |v^0|^2\Big)dx\\ \ecart\dis
        =\into\Big(\int_{\cal O}\Big({\rho^s\over 2}|\partial_t\hat u|^2+{1\over 2}Ae_y(\widetilde u):e_y(\widetilde u)\Big)dy+{\rho^l\over 2} |v|^2\Big)dx\\ \ecart\dis
        +2\mu\int_0^t\into\Big(\gamma\int_{\RR^3\setminus{\cal O}}|e_y(\hat v)|^2dy+ |e_y(v)|^2\Big)dxds.
        \ea$$
    Thus, taking into account that for $R_\ep>0$ satisfying (\ref{proRe}), we get
    $$\ba{l}\dis \int_{\Omega_\ep^s} \Big({\rho^s\ep^3\over 2r_\ep^3}|\partial_t u_\ep|^2 +{\ep^3\over 2r_\ep} A e_x(u_\ep) : e_x(u_\ep) \Big) dx + \int_{\Omega_\ep^l} \Big({\rho^l\over 2} |v_\ep|^2 + 2\mu \int_0^t|e_x(v_\ep)|^2ds\Big)dx\\ \ecart\dis =
        \int_{\Om^\ast_\ep}\int_{\cal O}\Big({\rho^s\over 2}|\partial_t\hat u_\ep|^2+{1\over 2} Ae_y(\hat u_\ep):e_y(\hat u_\ep)\Big)dydx+{\rho^l\over 2}\int_{\Om_\ep^l}|v_\ep|^2dx\\ \ecart\dis+2\mu\int_0^t\int_{\Om\setminus B(\ep k,R_\ep)}\hskip-0.5cm|e_x(v_\ep)|^2dxds+{2\mu r_\ep\over\ep^3}\int_0^t \into\int_{B_{R_\ep/r_\ep}}\hskip-10pt|e_y(\hat v_\ep)|^2dydxds,
        \ea$$
    we have proved the following convergence result for the energies
    $$\ba{l}\dis\int_{\Om^\ast_\ep}\int_{\cal O}\Big({\rho^s\over 2}|\partial_t\hat u_\ep|^2+{1\over 2} Ae_y(\hat u_\ep):e_y(\hat u_\ep)\Big)dydx+{\rho^l\over 2}\int_{\Om_\ep^l}|v_\ep|^2dx\\ \ecart\dis+2\mu\int_0^t\int_{\Om\setminus B(\ep k,R_\ep)}\hskip-0.5cm|e_x(v_\ep)|^2dxds+{2\mu r_\ep\over\ep^3}\int_0^t \into\int_{B_{R_\ep/r_\ep}}\hskip-10pt|e_y(\hat v_\ep)|^2dydxds\\ \ecart\dis
        \longrightarrow \into\Big(\int_{\cal O}\Big({\rho^s\over 2}|\partial_t\hat u|^2+{1\over 2}Ae_y(\widetilde u):e_y(\widetilde u)\Big)dy+{\rho^l\over 2} |v|^2\Big)dx\\ \ecart\dis
        +2\mu\int_0^t\into\Big(|e_y(v)|^2+\gamma\int_{\RR^3\setminus{\cal O}}|e_y(\hat v)|^2dy\Big)dxds.
        \ea$$
    From this convergence it is simple to check that (\ref{cofhue}), (\ref{confve}) and (\ref{convhv}) hold.
\end{proof}
\begin{remark} \label{reelvm}
    Theorem \ref{thm:co2e} provides a two-scale limit problem for (\ref{Peps}) depending on the macroscopic variables $(t,x)$ and the microscopic on $y$. \par For $\gamma=0$ the equations describing the behavior of the elastic and fluid components are uncoupled.
    The macroscopic behavior of the fluid is just given  by (\ref{pblim2eg0b}). In this case the solid particles are so small that they do not affect the movement of the fluid.\par
    In the case $\gamma=\infty$ and $\gamma\in (0,\infty)$, the behavior of the solid particles and the fluid are interrelated. In order to get a macroscopic equation for the velocity and pressure of the fluid it is necessary to eliminate the function $\hat u$ in (\ref{pblim2egin}) and the functions $(\hat u, \hat v,\hat p)$ in (\ref{pblim2eg0in}). The corresponding result is given by Theorem \ref{enlgafi} below, which shows that the presence of solid particles leads to the appearance of a memory term in the Stokes system that governs the behavior of the fluid.
\end{remark}

\begin{theorem} \label{enlgafi} Assume $\gamma\in (0,\infty]$, then there exist ${\cal R}_\gamma\in L^2_{loc}([0,\infty);H^1({\cal O})')^3$ and
    ${\cal S}_\gamma\in L^2_{loc}([0,\infty);L^2({\cal O})')^3$ such that  $v$ and $\nabla_x p$ in Theorem \ref{thm:co2e} are characterized as the unique solution of the following Stokes problem with memory terms
    \beq\label{pblim2eg0innl}\left\{\ba{ll}\dis\partial_t \Big((2\rho^s|{\cal O}|+\rho^l)v+\int_0^t{\cal S}_\gamma(t-s)v\,ds\Big)+{\cal R}_\gamma(t) \hat u^0\\ \ecart\dis \hskip3cm+{\cal S}_\gamma(t)\hat u^1-2\mu\,{\rm div}_xe_x( v)+\nabla_xp=f&\dis\hbox{in }(0,T)\times\Om,\\ \ecart\dis
    {\rm div}_xv=0&\dis\hbox{in }(0,T)\times\Om,\\ \ecart\dis
    v=0&\dis\hbox{on }(0,T)\times\partial\Om,\\ \ecart\dis
    v_{|t=0}=v^0 &\dis\hbox{in }\Om.\ea\right.\eeq
    Moreover, the operators ${\cal S}_\gamma$ and ${\cal R}_\gamma$ satisfy the positivity condition
    \beq\label{positnl}\ba{l}\dis{1\over 2}\Big|
    2\rho^s|{\cal O}|w+u^1+\int_0^t\Big({\cal S}_\gamma(t-s)w
    +{\cal R}_\gamma u^0+{\cal S}_\gamma u^1\Big)ds\Big|^2\\ \ecart\dis \leq
    \rho^s|{\cal O}| \Bigg(\rho^s\big(|w|^2-|w^0|^2\big)+\int_0^t\Big(\partial_t\int_0^r{\cal S}_\gamma(r-s)w\,ds+{\cal R}_\gamma u^0+{\cal S}_\gamma u^1\Big)wdr\Bigg)\\ \ecart\dis
    +{\rho^s|{\cal O}|\over 2}\int_{\cal O}\big(\rho^s|u^1|^2+Ae_y(u^0):e_y( u^0)\big)dy,\ \ \forall\,t\in (0,T),
    \ea\eeq
    for every $w\in H^1(\RR)$, every $u^0\in H^1({\cal O})^3$ and every $u^1\in L^2({\cal O})^3$.\par
    For $\gamma<\infty$, the operators ${\cal R}_\gamma$ and
    ${\cal S}_\gamma$ are in fact in $ L^2(0,\infty;H^1({\cal O})')^3$ and
    ${\cal S}_\gamma\in L^2(0,\infty;L^2({\cal O})')^3$ respectively and there exists $C>0$ independent of $\gamma$ such that
    \beq\label{acnogf}\|{\cal R}\|_{L^2(0,\infty;H^1({\cal O})')^3}+\|{\cal S}\|_{L^2(0,\infty;L^2({\cal O}))^3}\leq C \sqrt{\gamma}.\eeq
\end{theorem}
\begin{proof} Let us first assume $\gamma\in (0,\infty).$\par
    The functions $(\hat u,\hat v,\hat p)$ in (\ref{pblim2eg0in}) can be obtained from $v$ by solving the system
    \beq\label{pbauxgf}\left\{\ba{ll}\dis
    \rho^s\partial^2_{tt}\hat u-{\rm div}_y(Ae_y(\hat u))=0&\dis\hbox{in }(0,T)\times \Om\times {\cal O},\\ \ecart\dis
    -2\mu\,{\rm div}_y e_y(\hat v)+\nabla_y\hat p=0&\dis\hbox{in }(0,T)\times\Om\times (\RR^3\setminus{\cal O}),\\ \ecart\dis
    {\rm div}_y\hat v=0&\dis\hbox{in }(0,T)\times\Om\times (\RR^3\setminus{\cal O}),\\ \ecart\dis
    \partial_t\hat u=\hat v,\ Ae_y(\hat u)\nu=\gamma\big(2\mu e_y(\hat v)-\hat pI\big)\nu&\dis\hbox{on }(0,T)\times\Om\times \partial{\cal O},\\ \ecart\dis
    \hat v-v\in D^{1,2}(\RR^3\setminus{\cal O})^3,\\ \ecart\dis
    \hat u_{|t=0}=\hat u^0,\ \partial_t\hat u_{|t=0}=\hat u^1&\dis\hbox{in }\Om\times {\cal O}.\ea\right.\eeq
    Let us use the semigroup theory to solve this problem, for such purpose we define ${\cal A}_\gamma:D({\cal A}_\gamma)\subset H^1({\cal O})^3\times L^2({\cal O})^3\to H^1({\cal O})^3\times L^2({\cal O})^3$ by
    $$\ba{l}\dis D({\cal A}_\gamma)=\Big\{(\widetilde u,\widetilde z)\in H^1({\cal O})^3\times L^2({\cal O})^3:\ {\rm div}_y Ae_y(\widetilde u)\in L^2({\cal O})^3,\ \widetilde z\in H^1({\cal O})^3,\\ \ecart\dis\qquad\qquad\qquad \exists (\widetilde v,\widetilde p)\in D^{1,2}(\RR^3\setminus {\cal O})^3\times L^2(\RR^3\setminus {\cal O})\hbox{ such that }\\ \ecart\dis\qquad\qquad\qquad \qquad
        -2\mu {\rm div}_y e_y(\widetilde  v)+\nabla_y\widetilde p=
        {\rm div}_y\widetilde  v=0 \hbox{ in } \RR^3\setminus{\cal O},\\ \ecart\dis\qquad\qquad\qquad \qquad\qquad
        \widetilde  z=\widetilde v,\ Ae_y(\widetilde  u)\nu=\gamma\big(2\mu e_y(\widetilde  v)-\widetilde p I\big)\nu\ \hbox{ on }\partial{\cal O}\Big\},
        \ea$$
    $${\cal A}_{\gamma}\left(\ba{c} \widetilde u\\ \dis \widetilde z\ea\right)=\left(\ba{c}\widetilde z\\ \dis{1\over \rho^s}{\rm div}_y Ae_y(\widetilde u)\ea\right),\ \ \forall(\widetilde u,\widetilde z)\in H^1({\cal O})^3\times L^2({\cal O})^3$$
    in such way that taking
    \beq\label{cavsgf}\widetilde u=\hat u-\int_0^tv\,ds,\quad \widetilde z=\partial_t\hat u-v,\quad \widetilde v=\hat v-v,\eeq
    we have that (\ref{pbauxgf}) is equivalent to
    \beq\label{pbauegf}{d\over dt}\left(\ba{c} \widetilde u\\ \dis \widetilde z\ea\right)={\cal A}_{\gamma}\left(\ba{c} \widetilde u\\ \dis \widetilde z\ea\right)-\left(\ba{c} 0\\ \dis \partial_tv\ea\right)\ \hbox{ in }(0,T),\quad \left(\ba{c} \widetilde u(0)\\ \dis \widetilde z(0)\ea\right)= \left(\ba{c} \hat u^0\\ \dis \hat u^1-v(0)\ea\right).
    \eeq
    Taking in $H^1({\cal O})^3\times L^2({\cal O})^3$ the norm
    $$\|(\widetilde u,\widetilde z)\|^2=\int_{\cal O}\Big(Ae_y(\widetilde u):e_y(\widetilde u)+|\widetilde u|^2+\rho^s|\widetilde z|^2\Big)dy,\ \ \forall\, (\widetilde u,\widetilde z)\in H^1({\cal O})^3\times L^2({\cal O})^3,$$
    and using that
    $$\left({\cal A}_{\gamma}\left(\ba{c} \widetilde u\\ \dis \widetilde z\ea\right),\left(\ba{c} \widetilde u\\ \dis \widetilde z\ea\right)\right)=\int_{\cal O}\widetilde z\cdot\widetilde u\,dy-2\mu\gamma\int_{\RR^3\setminus {\cal O}}\hskip-5pt |e_y(\widetilde v)|^2dy,\quad \forall\, (\widetilde u,\widetilde z)\in D({\cal A}_\gamma),$$
    we get that ${\cal A}_\gamma - I$ is dissipative.
    Moreover, reasoning similarly to the proof of Theorem \ref{Thexsol},
    the Lax-Milgram's theorem proves that for every $(g,h)\in H^1({\cal O})^3\times L^2({\cal O})^3$ and every $\lambda>0$ problem
    $${\cal A}_\gamma\left(\ba{c} \widetilde u\\ \dis \widetilde z\ea\right)-\lambda\left(\ba{c} \widetilde u\\ \dis \widetilde z\ea\right)=\left(\ba{c} g\\ \dis h\ea\right),$$
    or equivalently
    $$\left\{\ba{ll}\dis \widetilde z-\lambda \widetilde u=g& \hbox{in }{\cal O},\\ \ecart\dis
        {\rm div}\big(Ae_y(\widetilde u)\big)-\lambda\rho^s\widetilde z=\rho^s h&\dis\hbox{in }{\cal O},\\ \ecart\dis
        2\mu{\rm div}_y e_y(\widetilde v)-\nabla_y\widetilde p=0 &\dis\hbox{in }\RR^3\setminus{\cal O},\\ \ecart\dis
        {\rm div}_y\widetilde v=0&\dis\hbox{in }\RR^3\setminus{\cal O},\\ \ecart\dis
        \widetilde  z=\widetilde v,\ Ae_y(\widetilde  u)\nu=\gamma\big(2\mu e_y(\widetilde  v)-\widetilde p I\big)\nu&\dis\hbox{on }\partial{\cal O},\\ \ecart\dis \widetilde v\in D^{1,2}(\RR^3\setminus{\cal O})^3,
        \ea\right.$$
    has a unique solution, such that $(\widetilde u,\widetilde v)\in D({\cal A}_\gamma)$, $(\widetilde v,\widetilde p)\in D^{1,2}(\RR^3\setminus{\cal O})^3\times L^2(\RR^3\setminus{\cal O}))$. This proves that ${\cal A}_\gamma-I$ is in the conditions of the Lumer-Philips theorem and then that ${\cal A}_\gamma$ generates a continuous semigroup ${\cal G}(t)$ in $H^1({\cal O})^3\times L^2({\cal O})^3$. Writing
    $${\cal G}(t)=\left(\ba{cc} {\cal G}_{11}(t) & {\cal G}_{12}(t)\\ \ecart\dis
        {\cal G}_{21}(t) & {\cal G}_{22}(t)\ea\right),$$
    where the definition of ${\cal A}_\gamma$ gives
    \beq\label{exdpsg}{\cal G}_{21}(t)\widetilde u={d{\cal G}_{11}\over dt}(t)\widetilde u,\ \  {\cal G}_{22}(t)\widetilde z={d{\cal G}_{12}\over dt}(t)\widetilde z,\quad \forall(\widetilde u,\widetilde z)\in H^1({\cal O})^3\times L^2({\cal O})^3,\eeq
    we get by (\ref{cavsgf}) and (\ref{pbauegf}) that the function $\hat u$ in (\ref{pbauxgf}) is given by
    \beq\label{exphu}\ba{ll}\dis \hat u&\dis=\int_0^tvds+{\cal G}_{11}(t)\hat u^0+{\cal G}_{12}(t)(\hat u^1-v(0))-\int_0^t {\cal G}_{12}(t-s)\partial_svds\\ \ecart &\dis= \int_0^tvds+{\cal G}_{11}(t)\hat u^0+{\cal G}_{12}(t)\hat u^1-\int_0^t {d{\cal G}_{12}\over dt}(t-s)vds.\ea\eeq
    Let us obtain some regularity properties for ${\cal G}$. We have that for every $(\widetilde u^0,\widetilde u^1)\in H^1({\cal O})^3\times L^2({\cal O})^3$ the function
    $$\widetilde u={\cal G}_{11}(t)\widetilde u^0+{\cal G}_{12}(t)\widetilde u^1$$
    is characterized by the existence of $(\widetilde v,\widetilde p)\in L^2(0,\infty;D^{1,2}(\RR^3\setminus{\cal O}))^3\cap L^2(0,\infty;L^2(\RR^3\setminus{\cal O}))$ such that $(\widetilde u,\widetilde v,\widetilde p)$ solves
    \beq\label{carCGgf}\left\{\ba{ll} \rho^s\partial^2_{tt}\widetilde u-{\rm div}_y\big(Ae_y(\widetilde u)\big)=0&\dis\hbox{in }(0,\infty)\times {\cal O},\\ \ecart\dis
    -2\mu{\rm div}e_y(\widetilde v)+\nabla_y\widetilde p=0&\dis\hbox{in }(0,\infty)\times (\RR^3\setminus{\cal O}),\\ \ecart\dis
    {\rm div}_y\widetilde v=0 &\dis\hbox{in }(0,\infty)\times (\RR^3\setminus{\cal O}),\\ \ecart\dis \partial_t\widetilde u=\widetilde v,\ Ae_y(\widetilde u)\nu=\gamma\big(2\mu e_y(\widetilde v)-\widetilde pI\big)\nu&\dis \hbox{on }(0,\infty)\times \partial{\cal O},\\ \ecart\dis
    \widetilde u_{|t=0}=\widetilde u^0,\ \partial_t\widetilde u_{|t=0}=\widetilde u^1&\dis\hbox{in }{\cal O}.\ea\right.\eeq
    For this system, the conservation of the energy reads as
    \beq\label{consegf}{1\over 2}{d\over dt}\int_{\cal O}\big(\rho^s|\partial_t\widetilde u|^2+Ae_y(\widetilde u):e_y(\widetilde u)\big)dy+2\gamma\mu\int_{\RR^3\setminus{\cal O}}|e_y(\widetilde v)|^2dy=0.\eeq
    \par
    On the other hand, for $\varphi\in C^\infty_c(\RR^3)$ such that $\varphi=1$ in ${\cal O}$, (\ref{carCGgf}) gives
    \beq\label{exdesgf}\ba{l}\dis \rho^s\partial^2_{tt}\int_{\cal O}\widetilde u\,dy=\int_{\partial{\cal O}}Ae_y(\widetilde u)\nu\,ds(y)=\gamma\int_{\partial{\cal O}}\big(2\mu e_y(\widetilde v)-pI\big)\nu\,ds(y)\\ \ecart\dis =-2\gamma\mu\int_{\RR^3\setminus{\cal O}}\hskip-4pt e_y(\widetilde v)\nabla_y\varphi\,dy+\gamma\int_{\RR^3\setminus{\cal O}}\hskip-4pt\widetilde  p\nabla_y\varphi\,dy.\ea\eeq
    In the last term of this equality we can use that thanks to De Rham theorem and $\nabla \widetilde p=2\mu {\rm div}e(\widetilde v)$ in $\RR^3\setminus{\cal O}$, there exists a constant $C>0$ and a measurable function $\widetilde c:(0,\infty)\to \RR$ such that
    $$\|p-\widetilde c\|_{L^2(\RR^3\setminus{\cal O})}\leq C\|e(\widetilde v)\|_{L^2(\RR^3\setminus {\cal O})^{3\times 3}}\ \hbox{ a.e. in }(0,\infty).$$
    Using then that $\varphi\in C^\infty_c(\RR^3\setminus {\cal O})$, $\varphi=1$ in ${\cal O}$ implies
    $$\int_{\RR^3\setminus{\cal O}}\hskip-4pt\nabla_y\varphi\,dy=-\int_{\partial{\cal O}}\nu\,ds(y)=0,$$
    the last term in (\ref{exdesgf}) satisfies
    $$\int_{\RR^3\setminus{\cal O}}\hskip-4pt \widehat p\nabla_y\varphi\,dy=\int_{\RR^3\setminus{\cal O}}\hskip-4pt \big(\widehat p-c\big)\nabla_y\varphi\,dy.$$
    Therefore (\ref{exdesgf}) implies for another constant $C>0$
    $$\Big|\int_{\cal O}\partial^2_{tt}\widetilde u\,dy\Big|\leq C\gamma\|e_y(\widetilde v)\|_{L^2(\RR^3\setminus{\cal O})^{3\times 3}}\ \hbox{ a.e. in }(0,\infty),$$
    which thanks to (\ref{consegf}) proves
    $$\int_{\RR}\Big|\int_{\cal O}\partial^2_{tt}\widetilde u\,dy\Big|^2dy\leq C\gamma\Big(\|e_y(\widetilde u^0)\|^2_{L^2({\cal O})^{3\times 3}}+\|\widetilde u^1\|^2_{L^2({\cal O})^3}\Big).$$
    This inequality and (\ref{consegf}) prove that  the linear operators ${\cal R}(t)\in (H^1({\cal O})^3)'$, ${\cal S}(t)\in L^2({\cal O})^3$ defined by
    \beq\label{decRcS}\langle {\cal R}(t),\widetilde u^0\rangle =\rho^s\partial^2_{tt}\int_{\cal O}{\cal G}_{11}(t)\widetilde u^0dy,\quad
    \int_{\cal O}{\cal S}(t)\widetilde u^1dy=\rho^s\partial^2_{tt}\int_{\cal O}{\cal G}_{12}(t)\widetilde u^1dy,
    \eeq
    for every $\widetilde u^0\in H^1({\cal O})^3$, $\widetilde u^1\in L^2({\cal O})^3$ satisfy
    ${\cal R}\in L^2(0,\infty;H^1({\cal O})')^3$, ${\cal S}\in L^2(0,\infty;L^2({\cal O}))^3$. Moreover, there exists $C>0$ such that (\ref{acnogf}) is satisfied.\par
    Deriving twice with respect to $t$ in  (\ref{exphu}) and using the first equation in (\ref{pblim2eg0in}), we deduce that  $v$, $p$ solve
    (\ref{pblim2eg0innl}).\par
    On the other hand, replacing in (\ref{pbauxgf}) $v$ by a function $w\in H^1(0,\infty)$ and $\hat u^0$, $\hat u^1$ as functions $u^0\in H^1({\cal O})^3$, $u^1\in L^2({\cal O})^3$, multiplying the first equation  by $\partial_t\hat u-w$, the second one by $\hat v-v$, integrating and adding, we get the energy identity
    $${1\over 2}{d\over dt}\int_{\cal O}\big(\rho^s|\partial_t\hat u|^2+Ae_y(\widetilde u):e_y(\widetilde u)\big)dy+2\mu\gamma \int_{\RR^3\setminus{\cal O}}|e_y(\hat v)|^2dy=\rho^s\int_{\cal O}\partial^2_{tt}\hat u\,dy\cdot v\ \hbox{ in }(0,T),$$
    which using (\ref{exphu}) as above can be written as
    $$\ba{l}\dis {1\over 2}{d\over dt}\int_{\cal O}\big(\rho^s|\partial_t\hat u|^2+Ae_y(\widetilde u):e_y(\widetilde u)\big)dy+2\mu\gamma \int_{\RR^3\setminus{\cal O}}|e_y(\hat v)|^2dy\\ \ecart\dis
        =\Big(\partial_t\Big(2\rho^s|{\cal O}|w+\int_0^t{\cal S}(t-s)w\,ds\Big)+{\cal R}(t) u^0+{\cal S}(t)u^1\Big)v.\ea.$$
    Integrating in $(0,T)$, using (\ref{exphu}) and the definitions of ${\cal R}$ and ${\cal S}$, and taking into account that
    $$\Big|\int_{\cal O}\partial_t \hat u(0,y)dy+\int_0^t\int_{\cal O}\rho^s\partial^2_{tt}\hat u\,dyds\Big|^2=\Big|\int_{\cal O}\partial_t\hat u(t,y)dy\Big|^2\leq |{\cal O}|\int_{\cal O}|\partial_t\hat u|^2dy,$$
    we deduce in particular (\ref{positnl})
    \par\medskip
    Let us now consider the case $\gamma=\infty$. The reasoning is similar (somewhat simpler). Now,  $\hat u$ is obtained from $v$, $\hat u^0$ and $\hat u^1$ by solving the system
    \beq\label{sisthugi}\left\{\ba{ll}\dis
    \rho^s\partial^2_{tt}\hat u-{\rm div}_y(Ae_y(\hat u))=0&\hbox{in }(0,T)\times \Om\times {\cal O}\\ \ecart\dis
    \partial_t\hat u=v &\dis\hbox{on }(0,T)\times\Om\times \partial{\cal O},\\ \ecart\dis
    \hat u_{|t=0}=\hat u^0,\ \ \partial_t\hat u_{|t=0}=\hat u^1 &\dis\hbox{in }\Om\times {\cal O},\ea\right.\eeq
    which implies that $\hat u$ is still given by (\ref{exphu}), where now  for every $(\widetilde u^0,\widetilde u^1)\in H^1({\cal O})^3\times L^2({\cal O})^3$ the function
    $$\widetilde u={\cal G}_{11}(t)\widetilde u^0+{\cal G}_{12}(t)\widetilde u^1$$
    is defined as the solution of
    \beq\label{desmggi}\left\{\ba{ll}\dis
    \rho^s\partial^2_{tt}\widetilde u-{\rm div}_y(Ae_y(\widetilde u))=0&\dis\hbox{in }(0,T)\times \Om\times {\cal O}\\ \ecart\dis
    \partial_t\widetilde u=0&\dis\hbox{on }(0,T)\times\Om\times \partial{\cal O},\\ \ecart\dis
    \widetilde u_{|t=0}=\widetilde u^0,\ \ \partial_t\widetilde u_{|t=0}=\widetilde u^1&\dis \hbox{in }\Om\times {\cal O},\ea\right.\eeq\par
    The conservation of the energy now reads as
    \beq\label{coegi}{1\over 2}{d\over dt}\int_{\cal O}\big(\rho^s|\partial_t\widetilde u|^2+Ae_y(\widetilde u):e_y(\widetilde u)\big)dy=0,\eeq
    while thanks to the hidden regularity theorem for the elasticity system (\cite{KuMaAm} , \cite{Lio}), we deduce that for every $t>0$, there exists $C_t>0$
    such that
    $$\ba{l}\dis \Big|\rho^s\int_0^t\int_{\cal O}\partial^2_{tt}\widetilde u\,dy\Big|^2=\Big|\int_0^t\int_{\partial{\cal O}}Ae_y(\widetilde u)\nu\,ds(y)\Big|^2\\ \ecart\dis \leq
        t|\partial{\cal O}|\int_{\partial{\cal O}}\big|Ae_y(\widetilde u)\nu\big|^2ds(y)\leq C_t\Big(\|\widetilde u^0\|^2_{H^1({\cal O})^3}+\|\widetilde u^1\|^2_{L^2({\cal O})^3}\Big).\ea$$
    This allows us to derive twice in (\ref{exphu}) to deduce that $u$ is still the solution of (\ref{pblim2eg0innl}) where now ${\cal R},$ ${\cal S}$ are just in
    $L^2_{loc}([0,\infty);H^1({\cal O}')^3)$ and $L^2_{loc}([0,\infty);L^2({\cal O})^3)$ respectively. Thanks to (\ref{coegi}) they are such that there exists $C>0$ satisfying
    \beq\label{acoiogi} \Big\|\int_0^t{\cal R}ds\Big\|_{(H^1({\cal O})')^3}+\Big\|\int_0^t{\cal S}ds\Big\|_{L^2({\cal O})^3}\leq C,\ \ \forall\, t\geq 0.\eeq
    Inequality (\ref{positnl}) follows as above, using the conservation of the energy for (\ref{sisthugi}).
\end{proof}

\begin{remark}[Explicit form of the memory kernels for $\gamma=\infty$]\label{rem:spectral}
    In the case $\gamma=\infty$, the microscopic system (\ref{sisthugi}) is a purely conservative elastodynamic equation with homogeneous Dirichlet boundary condition on the velocity. This allows us to compute the memory kernels ${{\cal R}_\infty}$ and ${{\cal S}_\infty}$ explicitly via spectral decomposition.

    Let ${\cal L}_D$ denote the Dirichlet elastic operator on ${\cal O}$, defined by
    $${\cal L}_D\phi=-{1\over \rho^s}{\rm div}_y(Ae_y(\phi)),$$
    with domain $D({\cal L}_D)=\{\phi\in H^1_0({\cal O})^3:\ {\rm div}_y(Ae_y(\phi))\in L^2({\cal O})^3\}$.
    Since ${\cal L}_D$ is a positive self-adjoint operator with compact resolvent, there exists an orthonormal basis $(\psi_n)_{n\geq 1}$ of $L^2({\cal O})^3$ consisting of eigenfunctions of ${\cal L}_D$ associated with positive eigenvalues $(\mu_n)_{n\geq 1}$ satisfying $\mu_n\to\infty$. We define the moments
    $$b_n:=\int_{\cal O}\psi_n(y)\,dy\in\RR^3,\quad n\geq 1,$$
    which satisfy $\sum_{n\geq 1}|b_n|^2=3|{\cal O}|$. Since $\hat u$ is determined by $v$, $\hat u^0$ and $\hat u^1$ through (\ref{sisthugi}), expanding in the basis $(\psi_n)_{n\geq 1}$ and integrating over ${\cal O}$ yields
    \begin{multline*}
        \rho^s\int_{\cal O}\partial^2_{tt}\hat u\,dy=\rho^s\sum_{n\geq 1}\Big(-\mu_n\cos(\sqrt{\mu_n}\,t)\,z^0_n(x)\\ -\sqrt{\mu_n}\sin(\sqrt{\mu_n}\,t)\,u^1_n(x)+\mu_n\int_0^t\cos(\sqrt{\mu_n}(t-s))\,v(s,x)\,ds\Big)b_n,
    \end{multline*}
    where
    $$z^0_n(x)=\int_{\cal O}(\hat u^0(x,y)-U^0(x,y))\cdot\psi_n(y)\,dy,\quad u^1_n(x)=\int_{\cal O}\hat u^1(x,y)\cdot\psi_n(y)\,dy,$$ and $U^0$ is the harmonic lifting of the boundary trace of $\hat u^0$.

    Comparing with (\ref{pblim2eg0innl}), the convolution kernel ${\cal S}_\infty$ acts as
    $${\cal S}_\infty(t)\,\xi=-\rho^s\sum_{n\geq 1}\mu_n\cos(\sqrt{\mu_n}\,t)\,(\xi\cdot b_n)\,b_n,\quad \forall\, \xi\in \RR^3,\quad t\geq 0.$$
    In particular, the matrix entries of ${\cal S}_\infty$ are
    $$\big({\cal S}_\infty(t)\big)_{ij}=-\rho^s\sum_{n\geq 1}\mu_n\cos(\sqrt{\mu_n}\,t)\,(b_n)_i(b_n)_j.$$
    Since each entry is an absolutely convergent trigonometric series with frequencies $\{\sqrt{\mu_n}\}_{n\geq 1}$, ${\cal S}_\infty$ is an \emph{almost-periodic} function of $t$. In particular, ${\cal S}_\infty(t)$ does not converge to zero as $t\to+\infty$ (unless ${\cal S}_\infty\equiv 0$, which is excluded since ${\cal S}_\infty(0)=-\rho^s\sum_n\mu_n(b_n)_i(b_n)_j\neq 0$). Nevertheless, the time-average of ${\cal S}_\infty$ vanishes:
    $${1\over\tau}\int_0^\tau {\cal S}_\infty(t)\,dt\to 0\quad \hbox{in }{\cal L}(\RR^3;\RR^3)\hbox{ as }\tau\to+\infty.$$
    Similarly, ${\cal R}_\infty(t)$ is almost-periodic in $t$ and, in general, does not decay as $t\to\infty$.

    This oscillatory, non-decaying character of the memory kernel is in stark contrast with the exponentially decaying kernels obtained in \cite{Cas2}, where the viscous dissipation inside the fluid inclusions acts as a damping mechanism. In the present case $\gamma=\infty$, the inclusions are purely elastic and isolated, so there is no microscopic dissipation and the memory persists indefinitely.
\end{remark}

\begin{remark}[Behavior of the memory kernels for $\gamma\in(0,\infty)$]\label{rem:damped}
    For finite $\gamma\in(0,\infty)$, the semigroup ${\cal G}(t)$ generated by ${\cal A}_\gamma$ governs the microscopic system (\ref{carCGgf}), in which the elastic oscillations in ${\cal O}$ are coupled to the exterior Stokes problem in $\RR^3\setminus{\cal O}$. The energy identity (\ref{consegf}) shows that the total microscopic energy is dissipated by the viscous term:
    $${1\over 2}{d\over dt}\int_{\cal O}\big(\rho^s|\partial_t\widetilde u|^2+Ae_y(\widetilde u):e_y(\widetilde u)\big)dy=-2\gamma\mu\int_{\RR^3\setminus{\cal O}}|e_y(\widetilde v)|^2dy\leq 0.$$
    As a consequence, the memory kernels ${\cal R}_\gamma$ and ${\cal S}_\gamma$ defined by (\ref{decRcS}) satisfy (\ref{acnogf}), which implies that they are small for $t$ large. Therefore, in contrast with the case $\gamma=\infty$ (Remark~\ref{rem:spectral}), the memory kernels for finite $\gamma$ go  to zero at infinity. The physical interpretation is clear: the viscous exterior Stokes flow acts as a damping mechanism on the elastic oscillations of the inclusions, producing a fading effective memory.
\end{remark}
\par\bigskip\noindent
{\bf Acknowledgments:}
This work  has been partially supported by FEDER project PID2023-149186NB-I00 and the IMUS-Maria de Maeztu grant CEX2024-001517-M - Apoyo a Unidades de Excelencia Mar\'ia de Maeztu. Both projects are funded by the {\it Ministerio de Ciencia, Innovaci\'on y Universidades} of the government of Spain.

\end{document}